\documentclass[1p]{elsarticle}

\usepackage{enumitem}

\usepackage{pdfpages,float}

\usepackage[T1]{fontenc}

\usepackage{amsfonts,latexsym,amssymb,amsmath,amsthm,mathrsfs}
\usepackage{color}

\usepackage{thmtools}
\usepackage[colorlinks, citecolor=blue, hypertexnames=false]{hyperref}

\declaretheorem[name=Theorem,numberwithin=section]{thm}
\newtheorem{lem}[thm]{Lemma}
\newtheorem{prop}[thm]{Proposition}

\theoremstyle{definition}
\newtheorem{defn}[thm]{Definition}
\newtheorem{rem}[thm]{Remark}
\newtheorem{claim}[thm]{Claim}

\numberwithin{equation}{section}
\numberwithin{table}{section}
\numberwithin{figure}{section}

\newcommand{\e}{\varepsilon}

\newcommand{\al}{\alpha}

\newcommand{\Om}{\Omega}
\newcommand{\om}{\omega}

\newcommand{\R}{\mathbb{R}}

\newcommand{\N}{\mathbb{N}}

\newcommand{\D}{\mathscr{D}}

\newcommand{\Oo}{\mathcal{O}}

\newcommand{\Cc}{C_{\e}}

\newcommand{\bmo}{{\rm BMO}}

\newcommand{\intav}{-\!\!\!\!\!\!\int}
\newcommand{\intq}{\frac{1}{\mu(Q)} \, \int_Q \,}
\newcommand{\intqj}{\frac{1}{\mu(Q_j)} \, \int_{Q_j} \,}

\DeclareMathOperator{\supp}{supp}

\def\ls{\lesssim}
\def\gs{\gtrsim}

\def\r{\right}
\def\lf{\left}

\def\noz{\nonumber}

\DeclareMathOperator{\dia}{diam}

\allowdisplaybreaks

\begin{document}
\begin{frontmatter}

\title{Functions of bounded mean oscillation  \\ and  quasisymmetric mappings \\ on spaces of homogeneous type\tnoteref{mytitlenote}}
\tnotetext[mytitlenote]{The first author was supported by an
Australian Government Endeavour Postgraduate Scholarship and a Lift-off Fellowship of the Australian Mathematical Society. The second author was supported by the Australian Research Council,
Grant ARC-DP160100153.}


\author[mymainaddress]{Trang T.T. Nguyen\corref{mycorrespondingauthor}}\cortext[mycorrespondingauthor]{Corresponding author}
\ead{trang.t.nguyen1@mymail.unisa.edu.au}

\author[mymainaddress]{Lesley A. Ward}
\ead{lesley.ward@unisa.edu.au}

\address[mymainaddress]{School of Information Technology and Mathematical Sciences\\ University of South Australia\\ Mawson Lakes SA 5095, Australia}

\begin{abstract}
We establish a connection between the function space~$\bmo$ and the theory of quasisymmetric mappings on \emph{spaces of homogeneous type}~$\widetilde{X} :=(X,\rho,\mu)$.
The connection is that the logarithm of the generalised Jacobian of an $\eta$-quasisymmetric mapping~$f: \widetilde{X} \rightarrow \widetilde{X}$ is always in~$\bmo(\widetilde{X})$.
In the course of proving this result, we first show that on~$\widetilde{X}$, the logarithm of a reverse-H\"{o}lder weight~$w$ is in~$\bmo(\widetilde{X})$, and that the above-mentioned connection holds on metric measure spaces~$\widehat{X} :=(X,d,\mu)$.
Furthermore, we construct a large class of spaces~$(X,\rho,\mu)$ to which our results apply.
Among the key ingredients of the proofs are suitable generalisations to $(X,\rho,\mu)$ from the Euclidean or metric measure space settings of the Calder\'{o}n--Zygmund decomposition, the Vitali Covering Theorem, the Radon--Nikodym Theorem, a lemma which controls the distortion of sets under an $\eta$-quasisymmetric mapping,  and a result of Heinonen and Koskela which shows that the volume derivative of an $\eta$-quasisymmetric mapping is a reverse-H\"{o}lder weight.
\end{abstract}


\end{frontmatter}

\newpage
\tableofcontents
\section{Introduction and Statement of Main Results}
\setcounter{equation}{0}
There is now an extensive body of research related to quasiconformal and quasisymmetric mappings on Euclidean spaces~$\R^n$. For instance, among the 33 problems about mappings, measures and metrics listed in~\cite{HS97}, eight are related to quasiconformal and quasisymmetric mappings, and two are specifically related to the Jacobian of a quasiconformal map. See also \cite{Chr07,Leo05}.

Recently, much effort has been devoted to extending the theory on Euclidean spaces to more general settings, in particular, metric measure spaces and spaces of homogeneous type.
As stated in~\cite{HKST15}, the origins of analysis on metric spaces ``lie in the search for an abstract context suitable to recover a substantial component of the classical Euclidean geometric function theory
associated to quasiconformal and quasisymmetric mappings''.

The present work is a contribution to this effort, where we establish a connection between the function space~$\bmo$ and the theory of quasiconformal mappings in an abstract and general setting, namely spaces of homogeneous type~$(X,\rho,\mu)$.
The connection is that the logarithm of the generalised Jacobian of an $\eta$-quasisymmetric mapping~$f: X \rightarrow X$ is always in~$\bmo(X)$.
This generalises a result of H.M. Reimann from the setting of Euclidean spaces~$\R^n$.
Reimann proved in \cite[Theorem~1]{Rei74} that the logarithm of the Jacobian determinant of a quasiconformal  mapping $f: \mathbb{R}^n \rightarrow \mathbb{R}^n$ is always in \( \bmo(\R^n) \).
We also note that the proof strategy of our results is completely independent of that of the original proof of Reimann's Theorem~1 on~$\R^n$.

In fact, a part of our goal is to identify a set of hypotheses on metric measure spaces~$(X,d,\mu)$ and spaces of homogeneous type~$(X,\rho,\mu)$, under which we can generalise Reimann's Theorem~1. As shown in the statements of our main results below, the spaces~$(X,d,\mu)$ and~$(X,\rho,\mu)$ are assumed to satisfy some additional conditions. Most of these conditions come from~\cite[Theorem~7.11]{HK98}, cited below as Theorem~\ref{7.11}, as the proof of our Theorem~\ref{thm.R1.metric} relies on this theorem. We will make clear in Remark~\ref{rem:hypo_use} and in our proofs where the hypotheses on the spaces are used. 

This paper has four main components:
(1) show that the logarithm of a reverse-H\"{o}lder weight on a space of homogeneous type is in~$\bmo$ (see Theorem~\ref{thm.logBMO} below),
(2) generalise Reimann's Theorem~1 to metric measure spaces (Theorem~\ref{thm.R1.metric}),
(3) generalise Reimann's Theorem~1 to spaces of homogeneous type (Corollary~\ref{thm.R1.quasimetric}), and
(4) construct a large class of spaces of homogeneous type to which our Corollary~\ref{thm.R1.quasimetric} applies (Example~\ref{thm:construct}). We describe these components in more detail below.

\( \bmo \) is the space of functions of \emph{bounded mean oscillation}.
A locally integrable real-valued function is in~$\bmo(\R^n)$ if its mean oscillation over all cubes in~$\R^n$ is uniformly bounded (Definition~\ref{def:BMO_R}).
The function space \( \bmo \) was first introduced by F. John, in his studies on rotation and strain in solid objects, in 1961 \cite{Joh}.
Since then, \( \bmo \) has been used in many different contexts.
\( \bmo \) also plays a key role in interpolation theorems used to establish the boundedness of operators on $L^p(\mathbb{R}^n)$, which in turn has applications in partial differential equations. See~\cite{Deco}, \cite{FS}, \cite{Gar81}, \cite{Joni}, \cite{Ste93}  and the references therein  for properties and more applications of~$\bmo$.

While conformal mappings take infinitesimal circles to infinitesimal circles, \emph{quasiconformal} mappings take infinitesimal circles to infinitesimal ellipses of uniformly bounded eccentricity (Definition~\ref{def:quasiconformal}).
Roughly speaking, at small scales, quasiconformal mappings can only distort shapes by a bounded amount.
They were introduced by Gr\"{o}tzsch (1928) and named by Ahlfors (1935).
Quasiconformal mappings found applications in various contexts, especially in complex analysis. See~\cite{Ahl06} for more details of quasiconformal mappings.
See also Section~\ref{subsec:quasisymmetic}.

A \emph{space of homogeneous type} is defined to be a triple~$(X,\rho,\mu)$, where~$X$ is a set, $\rho$ is a quasimetric on~$X$, and $\mu$ is a doubling measure on~$X$ (Definition~\ref{defn:(X,rho,mu)}).
Spaces of homogeneous type were introduced by Coifman and Weiss in 1971~\cite{CW71}.
Meyer wrote: ``\ldots~the action takes place today on spaces of homogeneous type. No group structure is available, the Fourier transform is missing, but a version of harmonic analysis is still present. Indeed the geometry is conducting the analysis'' \cite{DH09}.

Our first main result is an extension of a well known result in~$\R^n$. Given a reverse-H\"{o}lder weight~$w$ on a space of homogeneous type~$(X,\rho,\mu)$, we show that its logarithm is in~$\bmo$. We state this result as Theorem~\ref{thm.logBMO} below.
\begin{restatable}[\textbf{Reverse-H\"{o}lder weights and $\bmo$}]{thm}{logBMO}\label{thm.logBMO}
Suppose~$(X,\rho,\mu)$ is a space of homogeneous type. Suppose also that the measure~$\mu$ is Borel regular.
Let~$w$ be a weight on~$X$ such that~$w \in RH_q(X,\rho,\mu)$ for some~$q \in (1,\infty)$.
Then~$\log w \in \bmo(X,\rho,\mu)$.
\end{restatable}

The proof of Theorem~\ref{thm.logBMO} is composed of five main steps, grouped into two parts: (a) a reverse-H\"{o}lder inequality implies Muckenhoupt's~$A_p$ condition (Steps~1--3), and (b) logarithms of~$A_p$ weights are in~$\bmo$ (Steps~4 and~5). In the setting of metric spaces, part~(b) is stated without proofs in~\cite[Theorem~3.2]{KS14a}, which refers to~\cite{Shu12} for proofs. At first sight, part~(a) seems to be done in both \cite[Theorem~4.4]{KK11} and \cite[Theorem~3.1 parts (6)--(9)]{KS14}, which again refers to~\cite{Shu12} for proofs. However, the results in~\cite{KK11} and~\cite{KS14} require some extra assumptions on either the weight or the underlying measure space. Specifically, in~\cite{KK11} the weight is assumed to be doubling, while in~\cite{KS14} the underlying space must satisfy the so-called \emph{annular decay property}.

Our Theorem~\ref{thm.logBMO} completely avoids the extra assumptions imposed in the earlier literature, which makes it a nontrivial generalisation of previous related results. This is achieved by making use of some recent tools of dyadic analysis in metric and quasimetric spaces, such as Calder\'{o}n--Zygmund decompositions and properties of reverse-H\"{o}lder weights and $A_p$ weights.

\medskip
\noindent \emph{Outline of proof of Theorem~\ref{thm.logBMO}.}
As usual, $(X,\rho,\mu)$ is a space of homogeneous type.
Here the function classes~$RH_q(X)$, $RH_q^{\D^t}(X)$, $A_p(X)$, $A_p^{\D}(X)$,  $\bmo(X)$ and $\bmo_{\D}(X)$ are defined in terms of the quasimetric~$\rho$ on~$X$. We could write for example~$RH_q(X,\rho,\mu)$ instead of~$RH_q(X)$, but for brevity we have chosen not to do so.

\begin{enumerate}
  \item Develop a version (Theorem~\ref{CZ}) of the Calder\'{o}n--Zygmund decomposition on $X$ in terms of dyadic cubes.
  \item Let $\{\D^t: t = 1,2,\ldots,T\}$ be a collection of adjacent systems of dyadic cubes in~$X$ (see Definition~\ref{cubesinX}, Theorem~\ref{klp1}). Show that if~$w \in RH_q(X)$ for some~$q \in (1,\infty)$, then $w \in RH_q^{\D^t}(X)$ for each~$t \in \{1,\ldots, T\}$. (See Lemma~\ref{lem:RHr_implies_dRHr}.)
  \item Let $\D$ denote any fixed system of dyadic cubes.
  Show that if~$w \in RH_q^{\D}(X)$ for some~$q \in (1,\infty)$, then $w \in A_p^{\D}(X)$ for some~$p \in (1,\infty)$.
  (See Theorem~\ref{thm3}.)
  \item Let $\D$ denote any fixed system of dyadic cubes. Show that
  if~$w \in A_p^{\D}(X)$ for some~$p \in (1,\infty)$, then~$\log w \in \bmo_{\D}(X)$. (See Theorem~\ref{AplogW}.)
  \item Since $\bmo(X) = \cap^T_{t=1}\bmo_{\mathscr{D}^t}(X)$ (see~\cite[Proposition 7.16]{HK12}), conclude that $\log w \in \bmo(X)$.
  \end{enumerate}

Step~1 is done in Section~\ref{subsec:CZcubes}, Theorem~\ref{CZ}.
Step~2 is done in Section~\ref{subsec:RHr_imply_dRHr}, Lemma~\ref{lem:RHr_implies_dRHr}.
Step~3 is done in Section~\ref{subsec:RhripliesAp}, Theorem~\ref{thm3}.
Step~4 is done in Section~\ref{subsec:logAp_BMO}, Theorem~\ref{AplogW}.
Step~5 is straightforward.
Let $\{\D^t: t = 1,2,\ldots,T\}$ be a collection of adjacent system of dyadic cubes of $X$, as in Definition~\ref{cubesinX}. By Theorem~\ref{klp1}, such a collection exists. Fix $t \in \{1,\ldots,T\}$.
From Step~4 we obtain
$\log w \in \bmo_{\mathscr{D}^t}(X)$.
Since this is true for all $t=1,2,\ldots,T$, we conclude that
$\log w \in \bigcap_{t=1}^T \bmo_{\mathscr{D}^t}(X).$
By Proposition~7.16 in \cite{HK12}, $\log w \in \bmo(X)$ with
$\|\log w\|_{\bmo} \leq C\sum_{t=1}^T \|\log w\|_{\bmo_{\mathscr{D}^t}},$
where~$C > 0$ depends only on~$X$ and~$\mu$.
This together with the proofs given below of the results in Steps~1--4 completes the proof of Theorem~\ref{thm.logBMO}.
\(\hfill \Box\)
\medskip

A \emph{metric measure space} is defined to be a triple~$(X,d,\mu)$, where~$X$ is a set, $d$ is a metric on~$X$, and~$\mu$ is a doubling measure on~$X$. Note that every metric measure space is a space of homogeneous type.
With Theorem~\ref{thm.logBMO} in hand, we will be able to proceed with our main purpose, which is generalising Reimann's Theorem~1. This is done by applying Theorem~\ref{thm.logBMO} to specific weights, namely the generalised Jacobians $\widehat{J}_f$ and~$\widetilde{J}_f$.
Our second main result is stated in Theorem~\ref{thm.R1.metric}. We generalise Reimann's result from functions $f: \mathbb{R}^n \rightarrow \mathbb{R}^n$ to functions $f: (X,d,\mu) \rightarrow (X,d,\mu)$, where $(X,d,\mu)$ is a metric measure space satisfying certain conditions.
As the analogue of quasiconformality we use $\eta$-quasisymmetry. As the analogue of the Jacobian determinant~$J_f$, we use the generalised Jacobian~$\widehat{J}_f$ with respect to (w.r.t.) the metric~$d$, defined in equation~\eqref{eq:JacoQuasi_hat}.

\begin{restatable}[\textbf{Reimann's Theorem~1 generalised to $(X,d,\mu)$}]{thm}{metric}
\label{thm.R1.metric}
Let~$(X,d,\mu)$ be a metric measure space such that\\
\indent\textup{(i)} $\mu$ is a locally finite Borel-regular measure with dense support,\\
\indent \textup{(ii)} $X$ is locally compact,\\
\indent \textup{(iii)} $X$ is $\al$-regular for some~$\al > 1$, and\\
\indent \textup{(iv)} $X$ admits a weak $(1,p)$-Poincar\'{e} inequality for some~$p \in [1,\al)$.\\
Let $f$ be an $\eta$-quasisymmetric  mapping of~$(X,d,\mu)$ onto itself.
Let~$\widehat{J}_f$ be the generalised Jacobian of~$f$ w.r.t.~the metric~$d$.
Then~$\widehat{J}_f$ exists and is finite for~$\mu$-almost every ($\mu$-a.e.)~$x$, and $\log \widehat{J}_f \in \bmo(X,d,\mu)$.
\end{restatable}

We will show that~$\widehat{J}_f$ is a weight that satisfies the hypothesis of Theorem~\ref{thm.logBMO}. Specifically, $\widehat{J}_f$ is a reverse-H\"{o}lder weight. This requires our generalisation of the result of Heinonen and Koskela in~\cite[Theorem~7.1]{HK98}, and this is where conditions~(i)--(iv) on our space~$X$ are needed (see Theorems~\ref{7.11} and~\ref{thm:J_f in RHr}). Therefore~$\log \widehat{J}_f \in \bmo(X)$.

Theorem~\ref{thm.R1.metric} applies to a range of different metric measure spaces. To name a few: complete Riemannian manifolds~$X^n$, $n > 2$ with non-negative Ricci curvature, endowed with Riemannian volume and Riemannian metric \cite{HK98}; the metric measure space~$(H_1, \mathcal{L}_3)$ where $H_1 := (\R^3,d)$ is the first Heisenberg group equipped with the Carnot metric and~$\mathcal{L}_3$ denotes Lebesgue measure on~$\R^3$ \cite{Hei01}; the boundary of a hyperbolic building equipped with a geodesic metric and a Hausdorff measure \cite{BP99}, and the space~$F$ constructed as~$I \times K$, where~$I$ is an interval and $K$ is a Cantor set \cite{Laa00}.

The proof of Theorem~\ref{thm.R1.metric} is composed of four steps, which are outlined below.

%
%

\medskip
\noindent \emph{Outline of proof of Theorem~\ref{thm.R1.metric}.}
Here the function classes~$\bmo(X)$ and $RH_q(X)$ are defined in terms of the metric~$d$ on~$X$. We could write for example~$\bmo(X,d,\mu)$, but for brevity we have chosen not to do so.

\begin{enumerate}
  \item Show that~$\widehat{J}_f(x)$ exists and is finite for $\mu$-a.e. $x \in X$. (See Lemma~\ref{lem:J_f_hat exist}.)
  \item Show that~$\widehat{J}_f$ is measurable. (See Lemma~\ref{lem:Jfhat measurable}.) This is necessary to prove the next step.
  \item Show that $\widehat{J}_f \in RH_q(X)$. (See Theorem~\ref{thm:J_f in RHr}.)
  \item Using Theorem~\ref{thm.logBMO}, conclude that $\log \widehat{J}_f \in \bmo(X)$.
  \end{enumerate}
Step~1 is done in Section~\ref{subsec:Jfhat_exist}, Lemma~\ref{lem:J_f_hat exist}.
Step~2 is done in Section~\ref{subsec:Jfhat_meable}, Lemma~\ref{lem:Jfhat measurable}.
Step~3 is done in Section~\ref{subsec:Jfhat_RHr}, Theorem~\ref{thm:J_f in RHr}.
Step~4 is done in Section~\ref{subsec:log_Jfhat_BMO}.
\(\hfill \Box\)
\medskip

Can Reimann's result be extended even further to spaces of homogeneous type~$(X,\rho,\mu)$ with the generalised Jacobian being defined w.r.t.~a quasimetric~$\rho$, instead of a metric~$d$? The answer is yes.
Our third main result is stated in Corollary~\ref{thm.R1.quasimetric}. We generalise Reimann's result from functions $f: \mathbb{R}^n \rightarrow \mathbb{R}^n$ to functions $f: (X,\rho, \mu) \rightarrow (X,\rho, \mu)$, where $(X,\rho, \mu)$ is a space of homogeneous type satisfying certain conditions.
As the analogue of quasiconformality we again use $\eta$-quasisymmetry. As the analogue of the Jacobian determinant~$J_f$, we use the generalised Jacobian~$\widetilde{J}_f$ associated with the quasimetric~$\rho$, which is introduced in~Section~\ref{subsec:quasisymmetic}.

\begin{restatable}[\textbf{Reimann's Theorem~1 generalised to $(X,\rho,\mu)$}]{cor}{quasimetric}\label{thm.R1.quasimetric}
Suppose $\widetilde{X} :=(X,\rho,\mu)$ is a space of homogeneous type such that\\
\indent \textup{(a)} $\mu$ is a locally finite Borel-regular measure with dense support,\\
\indent \textup{(b)} $\widetilde{X}$ is locally compact, and\\
\indent \textup{(c)} $\widetilde{X}$ is $\al$-regular for some~$\al > 1$.\\
Given~$\e \in (0,1]$, let~$\rho_{\e}(x,y) := \rho(x,y)^{\e}$ for all~$x,y \in X$.
Let~$d_{\e}$ be a metric which is comparable to~$\rho_{\e}$.
Suppose the metric measure space~$\widehat{X} :=(X,d_{\e},\mu)$  satisfies\\
\indent \textup{(d)} $\widehat{X}$ admits a weak $(1,p)$-Poincar\'{e} inequality for some~$p\in [1,\al/\e)$.\\
Let $f$ be an $\eta$-quasisymmetric map from $\widetilde{X}$ onto itself.
Let~$\widetilde{J}_f$ be the generalised Jacobian of~$f$ w.r.t.~the quasimetric~$\rho$.
Then~$\widetilde{J}_f$ exists and is finite for $\mu$-a.e.~$x$, and~$\log  \widetilde{J}_f \in \bmo(\widetilde{X})$.
\end{restatable}
The metric~$d_{\e}$, which is comparable to~$\rho_{\e}$, can be obtained by using various approaches. In Section~\ref{subsec:spaceX}, we introduce three of them. Corollary~\ref{thm.R1.quasimetric} still holds if one uses any of those three, or even other approaches, to construct~$d_{\e}$.

The proof of Corollary~\ref{thm.R1.quasimetric} relies on Theorem~\ref{thm.R1.metric}, which is our generalisation of Reimann's Theorem~1 to metric measure spaces.
The idea of our proof is that starting from a space of homogeneous type~$\widetilde{X} :=(X,\rho,\mu)$ and an $\eta$-quasisymmetric map~$f$ on~$\widetilde{X}$ as in Corollary~\ref{thm.R1.quasimetric}, we can obtain a metric measure space~$\widehat{X} := (X,d_{\e},\mu)$ that satisfies the conditions of Theorem~\ref{thm.R1.metric}, and an $\varsigma$-quasisymmetric map~$f$ on~$\widehat{X}$.
Then by Theorem~\ref{thm.R1.metric}, we have~$\log  \widehat{J}_f \in \bmo(\widehat{X})$, where~$\widehat{J}_f$ is the generalised Jacobian associated with the metric~$d_{\e}$. Consequently, by the equivalence  of~$\widehat{J}_f$ and~$\widetilde{J}_f$ as shown in Lemma~\ref{lem:Jacoequi} and the fact that $\bmo(\widehat{X})$ and $\bmo(\widetilde{X})$ coincide as shown in Proposition~\ref{prop:BMO_coincide}, we conclude that~$\log  \widetilde{J}_f \in \bmo(\widetilde{X})$.

\medskip
\noindent \emph{Outline of proof of Corollary~\ref{thm.R1.quasimetric}.}
\begin{enumerate}
\item From the assumptions of Corollary~\ref{thm.R1.quasimetric} we obtain a metric measure space~$\widehat{X} :=(X,d_{\e},\mu)$ that satisfies the hypotheses of Theorem~\ref{thm.R1.metric}. In particular,\\
(i) $\mu$ is a locally finite Borel-regular measure with dense support,\\
(ii) $\widehat{X} $ is locally compact (see Lemma~\ref{lem:loccompequi}),\\
(iii) $\widehat{X} $ is $\al/\e$-regular for some~$\al > 1$ (see Lemma~\ref{lem:loccompactequi}), and\\
(iv) $\widehat{X} $ admits a weak $(1,p)$-Poincar\'{e} inequality for some~$p$ with $p \in [1,\al/\e)$.
\item Show that~$f$ is a $\zeta$-quasisymmetric map (w.r.t.~$d_{\e}$) from~$\widehat{X}$ onto itself for an appropriate homeomorphism~$\zeta: [0,\infty) \rightarrow [0,\infty)$. (See Lemma~\ref{lem:homeoequi}.)
\item Since Theorem~\ref{thm.R1.metric} holds for~$\widehat{X}$, we conclude that~$\widehat{J}_f(x)$ exists and is finite for $\mu$-a.e.~$x$ and~$\log  \widehat{J}_f \in \bmo(\widehat{X})$.
\item Show that $\widetilde{J}_f(x)$ exists and is finite for $\mu$-a.e.~$x$. (See Lemma~\ref{lem:J_f_tilde exist}.)
\item Show that~$\widehat{J}_f$ and~$\widetilde{J}_f$ are comparable. (See Lemma~\ref{lem:Jacoequi}.)
\item Using Steps~3--5, conclude that~$\log  \widetilde{J}_f \in \bmo(\widehat{X})$. (See Section~\ref{subsec:logJacoequi}.) 
\item Show that $\bmo(\widetilde{X}) = \bmo(\widehat{X})$. (See Proposition~\ref{prop:BMO_coincide}.) Hence~$\log  \widetilde{J}_f \in \bmo(\widetilde{X})$.
\end{enumerate}
Step~1 is done in Section~\ref{subsec:imply loccom}, Lemma~\ref{lem:loccompequi} and Section~\ref{subsec:imply alregular}, Lemma~\ref{lem:loccompactequi}.
Step~2 is done in Section~\ref{subsec:imply_quasisym}, Lemma~\ref{lem:homeoequi}.
Step~4 is done in Section~\ref{subsec:Jftilde exist}, Lemma~\ref{lem:J_f_tilde exist}.
Step~5 is done in Section~\ref{subsec: Jacoequi}, Lemma~\ref{lem:Jacoequi}.
Step~6 is done in Section~\ref{subsec:logJacoequi}.
Step~7 is done in Section~\ref{subsec:BMO_coincide}, Proposition~\ref{prop:BMO_coincide}. \(\hfill \Box\)
\medskip

A natural question to ask is whether there is any space of homogeneous type to which Corollary~\ref{thm.R1.quasimetric} applies. This will be our last main result. In Example~\ref{thm:construct}, we construct a large class of spaces of homogeneous type which satisfy the hypotheses of Corollary~\ref{thm.R1.quasimetric}.

\begin{restatable}[\textbf{Construction of suitable spaces of homogeneous type}]{exam}{construct}\label{thm:construct}
Let $(X, D, \mu)$  be a metric measure space such that\\
\indent \textup{(i)} $\mu$ is a locally finite Borel-regular measure with dense support,\\
\indent \textup{(ii)} $(X, D, \mu)$ is locally compact,\\
\indent \textup{(iii)} $(X, D, \mu)$ is $Q$-regular for some~$Q > 1$, and \\
\indent \textup{(iv)} $(X, D, \mu)$ admits a weak $(1,p)$-Poincar\'{e} inequality for some~$p \in [1,Q)$.\\
Fix~$\beta \geq 1$.
Define~$\rho (x,y) := D(x,y)^{\beta}$ for all~$x, y \in X$.
Then the space~$(X,\rho,\mu)$ satisfies the conditions~(a)--(d) of Corollary~\ref{thm.R1.quasimetric} with~$\e := 1/\beta$ and with~$\al := Q\e$ in condition~(c).
\end{restatable}
From the space of homogeneous type~$(X,\rho,\mu)$ constructed in Example~\ref{thm:construct}, we obtain a metric measure space~$(X,d_{\e},\mu)$ via the $\e$-chain approach.
In general, it is not known whether there is a nice characterisation of spaces of homogeneous type where the modified metric~$d_{\e}$ supports a weak $(1,p)$-Poincar\'{e} inequality.
Thus, besides providing a class of spaces of homogeneous type for which our Corollary~\ref{thm.R1.quasimetric} holds, our construction is also interesting in terms of addressing the issue mentioned above.

\medskip
\noindent \emph{Outline proof of Example~\ref{thm:construct}.}
The proof of Example~\ref{thm:construct} consists of four main steps, which are outlined below.
  \begin{enumerate}
  \item Start with the metric measure space~$(X, D, \mu)$ as stated in Example~\ref{thm:construct}.
  \item Fix~$\beta \geq 1$ and
define~$\rho (x,y) := D(x,y)^{\beta}$ for all~$x, y \in X$.
Show that~$\rho$ is a quasimetric with quasitriangle constant~$A_0 = 2^{\beta -1}$. (See Lemma~\ref{lem:construct1}.)
  \item Fix~$\e = 1/\beta$. Define~$d_{\e}(x,y)$ from~$\rho$ by the $\e$-chain approach.
  \item So far, we have constructed a space of homogeneous type~$(X,\rho,\mu)$ and a metric measure space~$(X,d_{\e},\mu)$. Finally, show that they satisfy the hypotheses of Corollary~\ref{thm.R1.quasimetric}.
\end{enumerate}
The detailed proofs of Lemma~\ref{lem:construct1} and Step~4 are presented in Section~\ref{sec:construction_X}.
\(\hfill \Box\)
\medskip

Besides the four main results explained above, we would like to emphasise some other results which are proved below in the setting of spaces of homogeneous type~$(X,\rho,\mu)$. They are not only used to prove our main result but are also of interest in their own right.
\begin{enumerate}
  \item Theorems~\ref{CZ} and~\ref{CZ2}: Calder\'{o}n--Zygmund decomposition,
  \item Lemma~\ref{thm3}: a dyadic reverse-H\"{o}lder weight is also a dyadic~$A_p$ weight,
  \item Lemma~\ref{lem:nu_dbling}: distortions of sets under an $\eta$-quasisymmetric map,
  \item Theorem~\ref{thm:vitali in X}: Vitali Covering Theorem, and
  \item Theorem~\ref{lem:Nages analog}: Radon--Nikodym Theorem.
\end{enumerate}

Throughout the paper, metrics are denoted by~$d$, $d_{\e}$ or~$D$ and quasimetrics are denoted by~$\rho$. (Metric) balls are denoted by~$\widehat{B}$ and quasiballs are denoted by~$\widetilde{B}$.
We use the usual notation
$\intav_{E} \,d\mu = \frac{1}{\mu(E)}\int_{E} \,d\mu$, where $E \subset X$.
Given a weight~$w \in L^1_{\text{loc}}(X)$, for each $\mu$-measurable subset~$E$ of~$X$ we define $w(E) := \int_{E} \,dw = \int_{E} w \,d\mu$.
We denote by~$C$ a positive constant that is independent of the main parameters but may vary from line to line.
If $f\le Cg$, we write $f\ls g$ or $g\gs f$; and if $f \ls g\ls
f$, we  write $f\sim g$, or~$f \sim_C g$ when we want to emphasise the constant.

This paper is organised as follows. In Section \ref{sec:defns}, we present the mathematical concepts needed later in the paper. This section includes proofs of some new results.
In Section~\ref{sec:logBMO}, we prove Theorem~\ref{thm.logBMO}.
In Section \ref{sec:quasi_result}, we present results which hold on spaces of homogeneous type. They include the results that the measure induced by a quasisymmetric map is doubling and Borel regular, the Vitali Covering Theorem and the Radon--Nikodym Theorem. These results will be used in the later sections.
In Section \ref{sec:R1_mec}, we prove Theorem~\ref{thm.R1.metric}, which is Reimann's Theorem 1 generalised to metric measure spaces $(X,d, \mu)$. In Section \ref{sec:R1_quasimec}, we prove Corollary~\ref{thm.R1.quasimetric}, which is Reimann's Theorem 1 generalised to  spaces of homogeneous type $(X,\rho, \mu)$. The construction of a large class of spaces of homogeneous type (Example~\ref{thm:construct}) for which our results hold is carried out in Section~\ref{sec:construction_X}.



\section{Background and Preliminaries}\label{sec:defns}
This section is organised as follows.
In Section~\ref{subsec:bmo_qsconformal}, we define two central concepts: the function space~$\bmo$ and quasiconformal mappings both in the Euclidean setting.
In Section \ref{subsec:spaceX}, we introduce metric measure spaces~$(X,d,\mu)$ and spaces of homogeneous type~$(X,\rho,\mu)$. In Section \ref{subsec:dyadiccubesss}, we explain systems of dyadic cubes and collections of adjacent systems of dyadic cubes.
In Section~\ref{subsec:dbl_dydbl_weight}, we introduce the concept of doubling and dyadic doubling weights.
In Section~\ref{subsec:BMO_X}, we define the function space~$\bmo$ on  metric measure spaces~$(X,d,\mu)$ and spaces of homogeneous type~$(X,\rho,\mu)$.
In Section \ref{subsec:quasisymmetic}, we define quasisymmetric maps defined on metric spaces~$(X,d)$ and quasimetric spaces~$(X,\rho)$ as well as their generalised Jacobians.
In Section~\ref{sec:weight}, we review the~$A_p$ weights and reverse-H\"{o}lder weights.
In Section~\ref{sec:meable func}, we define measurable functions and establish some of their properties.
In Section~\ref{sec:A_infty measure}, we discuss $A_{\infty}$ related measures and some of their properties.
For more detail on this material, see \cite{CW71}, \cite{HK98}, \cite{HK12}, \cite{KLPW16}, \cite{Tys98}, \cite{Geh73}, \cite{Hei01} and~\cite{Fol99}.
\subsection{The function space $\bmo$, and quasiconformal mappings, on~$\R^n$}\label{subsec:bmo_qsconformal}

\begin{defn}\label{def:BMO_R}
    A locally integrable real-valued function $f: \mathbb{R}^n
    \rightarrow \mathbb{R}$ is said to be of \emph{bounded mean
    oscillation}, written $f \in \bmo$ or $f \in \bmo(\mathbb{R}^n)$,
    if
    \[
        \|f\|_{\bmo}
        := \sup_Q\frac{1}{|Q|}\int_Q \lf|f(x) - f_Q\r| \,dx < \infty,
    \]
    where
    $f_Q = \frac{1}{|Q|}\int_Q f(y) \,dy$
    is the average of the function~$f$ over the cube~$Q$. Here~$Q$
    denotes a cube in $\mathbb{R}^n$ with edges parallel to the coordinate axes, and~$|Q|$ denotes the Lebesgue measure of~$Q$.
\end{defn}

\begin{defn}\label{def:quasiconformal}\cite{Rei74} (\textbf{Quasiconformal mapping})
A \emph{$K$-quasiconformal mapping} is a homeomorphism $f: G \rightarrow \mathbb{R}^n$ such that $f$ is absolutely continuous on lines, $f$ is totally differentiable almost everywhere, and there is a constant $K$ such that
\begin{equation}\label{0.1}
\sup_{\xi \in \mathbb{R}^n, |\xi| = 1}|F(x)\xi|^n \leq KJ_f(x) \text{ a.e.},
\end{equation}
where $G \subset \mathbb{R}^n$, $F(x)$ is the Jacobian matrix of $f$ at $x$ and $J_f(x)$ is the Jacobian determinant of $F(x)$.
\end{defn}

\subsection{Metric spaces~$(X,d,\mu)$ and spaces of homogeneous type~$(X,\rho,\mu)$}\label{subsec:spaceX}
In this section, we define metrics, quasimetrics and doubling measures, which let us define metric measure spaces and spaces of homogeneous type. We also describe some related concepts such as Borel-regularity, geometrically doubling, local compactness, $\al$-regularity and the $\e$-chain approach.

\begin{defn}(\textbf{Metric}) A \emph{metric} on a set $X$ is a function $d: X\times X \rightarrow [0, \infty)$ satisfying the following conditions for all $x, y, z\in X$:
\begin{eqnarray*}
  d(x,y) &=& 0 \quad \text{ if and only if } \quad x=y, \\
  d(x,y) &=& d(y,x), \\
  d(x,z) &\leq& d(x,y) + d(y,z).
\end{eqnarray*}
\end{defn}

The pair~$(X,d)$ is called a \emph{metric space}.
The metric~$d$ can be used to define balls, diameters of subsets of $X$, distances from a point to a subset and distances between subsets:
\begin{eqnarray*}
  \widehat{B}(x,r) &:=& \{y \in X: d(x,y) < r\}, \hspace{0.6cm} x \in X, r>0, \\
  \dia A &:=& \sup_{x,y \in A} d(x,y),\hspace{0.6cm} A \subset X, \\
  d(x,A) &:=& \inf_{y \in A} d(x,y), \hspace{0.6cm} x \in X, A \subset X, \\
  d(A,B) &:=&  \inf_{x \in A, y \in B} d(x,y), \hspace{0.6cm} A, B \subset X.
\end{eqnarray*}

A \emph{quasimetric} on a set~$X$ is a function~$\rho: X\times X \rightarrow [0, \infty)$ satisfying the same conditions as a metric, excepted that the triangle inequality is replaced by a \emph{quasitriangle inequality}:
\[\rho(x,z) \leq A_0\rho(x,y) + A_0d\rho(y,z),\]
where the \emph{quasitriangle constant} $A_0 \geq 1$ does not depend on $x, y$ or $z$.

The pair~$(X,\rho)$ is called a \emph{quasimetric space}.
As with a metric, a quasimetric can be used to define quasiballs~$\widetilde{B}(x,r)$, diameters~$\dia A$ of subsets of $X$, distances~$\rho(x,A)$ from a point to a subset and distances~$\rho(A,B)$ between subsets; here the metric~$d$ is replaced by the quasimetric~$\rho$.
In this paper we use the term \emph{quasiball} to mean a ball defined with respect to a quasimetric. To avoid any confusion, we note that in the literature the term quasiball can also denote the image of a ball under a quasiconformal or quasisymmtric map.



In addition to a metric, we need a doubling measure that is consistent with the chosen metric.
\begin{defn}\label{defn:dbl_measure} \cite{CW71} (\textbf{Doubling measure}) A \emph{doubling measure} on the space $(X, d)$ is a measure $\mu$  on $X$ such that the balls in $(X, d)$ are $\mu$-measurable sets, and the following condition holds for all $x \in X$ and all  $r > 0$:
\begin{equation}\label{eq:dbl_measure_1}
0 < \mu(\widehat{B}(x, 2r)) \leq A_1 \mu(\widehat{B}(x, r)) < \infty,
\end{equation}
where the \emph{doubling constant} $A_1 \geq 1$ does not depend on $x$ and $r$.
\end{defn}

In fact, inequality~\eqref{eq:dbl_measure_1} implies a more general property of the doubling measure~$\mu$. Namely, for all~$x \in X$, $r > 0$ and~$\lambda > 1$ we have
\begin{equation}\label{eq:dbl_measure_2}
    \mu(\widehat{B}(x, \lambda r)) \leq A_1^{1 + \log_2 \lambda} \mu(\widehat{B}(x, r)).
\end{equation}

A doubling measure on a quasimetric space~$(X,\rho)$ is defined in the same way, except the ball~$\widehat{B}$ is replaced by the quasiball~$\widetilde{B}$.
When a metric space~$(X,d)$ is equipped with a doubling measure~$\mu$, the triple~$(X,d,\mu)$ is called a \emph{metric measure space}.
When a quasimetric space~$(X,\rho)$ is equipped with a doubling measure~$\mu$, the triple~$(X,\rho,\mu)$ is called a \emph{space of homogeneous type}.

\begin{defn} \cite{CW71} (\textbf{Space of homogeneous type})\label{defn:(X,rho,mu)} A \emph{space of homogenous type} is a triple $(X, \rho, \mu)$ where $X$ is a nonempty set, $\rho$ is a quasimetric on $X$  and $\mu$ is a doubling measure on the space $(X, \rho)$.
\end{defn}
Following~\cite{Chr55}, we assume that the measure~$\mu$ is defined on a $\sigma$-algebra $\mathcal{M}$ which contains all Borel sets and all quasiballs~$\widetilde{B} \subset X$. A set~$E \subset X$ is $\mu$-\emph{measurable} if $E \in \mathcal{M}$.

Sometimes, we also require the measure~$\mu$ on the metric measure space~$(X,d,\mu)$ or on the space of homogeneous type~$(X,\rho,\mu)$ \emph{Borel regular}.
The measure~$\mu$ is Borel regular if for all Borel sets~$E \subset X$ we have
$$\mu(E)
 =\sup \{\mu(V): V \text{ closed}, V \subset E\}
 = \inf \{\mu(U): U \text{ open}, E \subset U\}.$$

We note that every space of homogeneous type is \emph{geometrically doubling} \cite{CW71}, meaning that there exists $N$ such that every quasiball $\widetilde{B}(x,r)$ can be covered by at most $N$ balls of radius~$r/2$.

\begin{defn}\label{def:compact}\cite{Mun74,Fol99} (\textbf{Locally compact spaces})
A metric space $(X,d)$ is \emph{locally compact} if for each $x \in X$ there exists an open set~$O$ w.r.t.~$d$ and a compact set~$K$ w.r.t.~$d$ such that $x \in O \subset K$.
\end{defn}

\begin{defn}\label{def:alpha}(\textbf{$\alpha$-regular spaces})
A metric space $X$ endowed with a doubling measure $\mu$ is an \emph{Ahlfors-regular space of dimension $\alpha$ }(for short, an \emph{$\alpha$-regular space}) if there exists a constant $\kappa \geq 1$ so that for every ball $\widehat{B}_r$ in $X$ with radius $r < \dia X$, we have
$\kappa^{-1}r^{\alpha} \leq \mu(\widehat{B}_r) \leq \kappa r^{\alpha}.$
\end{defn}
The notion of regularity of measures was first introduced by Ahlfors for $\al =1$ in~\cite{Ahl35}.
It was then explicitly identified and studied by David and Semmes for general~$\al$.
Some of the results of their study  related to Ahlfors regularity are collected in~\cite{DS93}.

Local compactness and $\al$-regularity for spaces of homogeneous type~$(X,\rho,\mu)$ are defined as in Definitions~\ref{def:compact} and~\ref{def:alpha} above, except that the metric~$d$ and the ball~$\widehat{B}_r$ are replaced by the quasimetric~$\rho$ and the quasiball~$\widetilde{B}_r$, respectively.

Given a quasimetric~$\rho$, it turns out that one can construct an metric~$d_{\e}$, depending on a constant~$\e \in (0,1]$, which is comparable to~$\rho_{\e}$, where~$\rho_{\e}(x,y) := \rho(x,y)^{\e}$ for all~$x,y \in X$.
That is, there exists a constant~$C_{\e}$ independent of~$x$ and~$r$ such that for all~$x, y \in X$ we have
\begin{equation}\label{eq:Hei1}
  \Cc^{-1}\rho_{\e}(x,y) \leq d_{\e}(x,y) \leq \Cc\rho_{\e}(x,y).
\end{equation}
The question of finding an appropriate~$\e$ such that~$\eqref{eq:Hei1}$ holds has been investigated by a number of authors. For example, in the proof of Theorem~2 in~\cite{MS79}, it is shown that~$\e$ can be chosen such that~$(3A_0^2)^{\e} = 2$, where~$A_0 \geq 1$ is the quasitriangle constant.
In the proof of Proposition~14.5 in~\cite{Hei01}, $\e$ can be chosen so that~$(2A_0)^{2\e} \leq 2$. In~\cite[Section~2]{PS09}, $\e$ is given by~$(2A_0)^{\e} = 2$.
We describe the construction in~\cite{PS09}, as it will be used in Section~\ref{sec:construction_X}. In~\cite{PS09}, $d_{\e}$ is produced via the so-called~\emph{$\e$-chain approach}.
\begin{defn}\label{def:e chain}
Let~$(X,\rho)$ be a quasimetric space.
Let~$\rho_{\e}(x,y) := \rho(x,y)^{\e}$ for all~$x,y \in X$.
Given~$\e \in (0,1]$,
define the function~$d_{\e}: X \times X \rightarrow [0,\infty)$ by
\begin{equation}\label{eq:pchain}
  d_{\e}(x,y) := \inf \bigg\{ \sum_{i=0}^{n} \rho_{\e}(x_{i},x_{i+1}): x = x_0,x_1, \ldots, x_n = y, \quad n \geq 1 \bigg\},
\end{equation}
The above process of producing~$d_{\e}$ from~$\rho$ is called the~\emph{$\e$-chain approach}.
\end{defn}
Recall that~$\rho_{\e}$ is also known as the \emph{snowflaking} of the quasimetric~$\rho$.
With~$\e$ chosen properly, $d_{\e}$ becomes a metric, and is comparable to the snowflaking~$\rho_{\e}$.
\begin{thm}\cite[Section~2]{PS09}\label{prop:chain_apporoach}
    Let~$(X,\rho)$ be a quasimetric space and let~$\e$ such that~$0< \e \leq 1$ be determined by~$(2A_0)^{\e} = 2$, where~$A_0$ is the quasitriangle constant. Then the function~$d_{\e}$ obtained from~$\rho$ by the $\e$-chain approach is a metric on~$X$ and is comparable to~$\rho_{\e}$.
\end{thm}

\begin{defn}\label{Xspace}
 We say a space of homogeneous type $(X,\rho,\mu)$ has \emph{nonempty $\tau$-annuli} if there exists $\tau \in (0,1)$ such that
 for each $x \in X$ and for each $r>0$, $\widetilde{B}(x,r)\backslash \widetilde{B}(x,\tau r)$ is nonempty whenever $X \backslash \widetilde{B}(x,r)$ is nonempty.
\end{defn}
The notion of nonempty  $\tau$-annuli is also called \emph{uniform perfectness} in standard literature on topology, dynamics, geometric measure theory, and analysis in metric spaces.

If $(X,\rho,\mu)$ has nonempty $\tau$-annuli, it follows from Definition~\ref{Xspace} that for all $\widetilde{B}(x,r) \subset X$ with $X \backslash \widetilde{B}(x,r) \neq \emptyset$ we have $\sup_{y \in \widetilde{B}(x,r)} \rho(x,y) > \tau r.$

\begin{prop}\label{prop:nonempty_annuli}
  If~$(X,\rho,\mu)$ is $\al$-regular with constant~$\kappa$, then~$X$ has nonempty $\tau$-annuli property for all~$\tau \in (0,\kappa^{-2/\al})$.
\end{prop}
\begin{proof}
  Take~$\tau \in (0,\kappa^{-2/\al})$. Note that $\tau < \kappa^{-2/\al}$ implies that~$\kappa^{-1} - \kappa \tau^{\al} > 0$.
Take~$\widetilde{B}(x,r) \subset X$. Then
\[\kappa^{-1} r^{\al} \leq \mu(\widetilde{B}(x,r)) \leq \kappa r^{\al}, \quad \text{and} \quad \kappa^{-1}\tau^{\al} r^{\al} \leq \mu(\widetilde{B}(x,\tau r)) \leq \kappa \tau^{\al}r^{\al}.\]
Thus Proposition~\ref{prop:nonempty_annuli} follows as
\begin{equation*}
  \mu(\widetilde{B}(x,r)\backslash \widetilde{B}(x,\tau r)) = \mu(\widetilde{B}(x,r))- \mu( \widetilde{B}(x,\tau r)) \geq
   (\kappa^{-1} - \kappa \tau^{\al})r^{\al} > 0.\qedhere
\end{equation*}
\end{proof}

On a metric space~$(X,d)$,
a curve $\gamma: I \rightarrow X$ is \emph{rectifiable} if its length is finite:
$l(\gamma) := \sup l(\gamma|J) < \infty,$
where the supremum is taken over all closed subintervals $J$ of $I$.
Let $U$ be an open set in $X$ and let $u$ be an arbitrary real-valued function on $U$. We say that the Borel function $\varrho: U \rightarrow [0,\infty]$ is a \emph{very weak gradient (upper gradient)} of $u$ in $U$ if for each rectifiable curve $\gamma_{xy}$ joining two points $x$ and $y$ in $U$,
\[|u(x) - u(y)| \leq \int_{\gamma_{xy}} \varrho \,ds.\]

\begin{defn}\label{pc}\cite{HK98} (\textbf{Rectifiably connected})
The metric space~$(X,d)$ is called \emph{rectifiably connected} if every pair of points in $X$ can be joined by a rectifiable curve.
\end{defn}

\begin{defn}\label{pc}\cite{HK98} (\textbf{Poincar\'{e} inequality})
Let $p \geq 1$ be a real number. We say that a metric measure space $(X,d,\mu)$ admits a \emph{weak $(1,p)$-Poincar\'{e} inequality} if there are constants $C_p \geq 1$ and $C_0 \geq 1$ such that
\begin{equation}\label{Pointcare1}
  \intav_{B} |u-u_B| \,d\mu \leq C_p(\dia B)\bigg(\intav_{C_0B} \varrho^p \,d\mu\bigg)^{1/p},
\end{equation}
whenever $u$ is a bounded continuous function on a ball $C_0B$ and $\varrho$ is its very weak gradient there. The constants should be independent of $B$ and $u$.
\end{defn}

\begin{rem}\label{rem:recti_connect}
It is known that if a metric measure space~$(X,d,\mu)$ is complete and admits a weak $(1, p)$–Poincar\'{e} inequality,
then~$X$ is \emph{quasiconvex}, meaning there exists a constant~$C$ such that every pair of points in $x, y \in X$ can be joined by a rectifiable curve~$\gamma$ such that
$\ell(\gamma) \leq Cd(x, y)$ \cite{Che99,Kei03}.
This forces the space to be rectifiably connected.

Moreover, if~$X$ is assumed to be locally compact instead of being complete, then~$X$ is still quasiconvex, and so rectifiably connected.
This is because in the proof of the quasiconvexity of~$X$, the assumption that~$X$ is complete is used only to guarantee that closed balls of finite radius are compact~\cite{Che99}. This can also be achieved by the local compactness of~$X$.

We will use this remark later to invoke the result of~\cite{HK98} that we use in our proofs.
\end{rem}

\subsection{Dyadic cubes in~$(X,\rho)$}\label{subsec:dyadiccubesss}
Since the proofs of Reimann's theorems involve the use of cubes in $\R^n$, we need an equivalent theory of cubes in quasimetric spaces~$(X,\rho)$.
In this section, we recall the construction systems of dyadic cubes, adjacent systems of dyadic cubes and their related properties.
This construction is originally developed in~\cite{HK12}. We present here the (slightly reworded) version that appears in~\cite[Section~2]{KLPW16}.
For the history of the development of systems of dyadic cubes, and collection of such systems which generalise the ``one-third trick'', see \cite{HK12} and the references therein, especially \cite{Chr55} and \cite {SaWh}.


\begin{defn}\label{dyadiccubes}\cite{KLPW16} (\textbf{A system of dyadic cubes}) In a geometrically doubling quasimetric space $(X, \rho)$, a countable family
\[\mathscr{D} = \bigcup_{k \in \mathbb{Z}} \mathscr{D}_k, \hspace{0.6cm} \mathscr{D}_k = \{Q_{\alpha}^k: \alpha \in \mathscr{A}_k\},\]
of Borel sets $Q_{\alpha}^k \subset X$ together with a fixed collection of countably many points $x_{\alpha}^k$ in $X$, with $x_{\alpha}^k \in Q_{\alpha}^k$ for each $k \in \mathbb{Z}$ and each $\alpha \in \mathscr{A}_k$,
is called a \emph{system of dyadic cubes with parameters $\delta \in (0,1)$ and $c_1$ and $C_1$} such that $0 < c_1 < C_1 < \infty$ if it has the following properties.
\begin{align}
  &1. \text{ } X = \bigcup_{\alpha \in \mathscr{A}_k} Q_{\alpha}^k \text{ (disjoint union) for all } k \in \mathbb{Z}.\label{a1}\\
  &2. \text{ If }l \geq k, \text{ then either } Q_{\beta}^l \subset Q_{\alpha}^k \text{ or } Q_{\alpha}^k \cap Q_{\beta}^l = \emptyset.\hspace{5cm} \label{a2}\\
  &3. \text{ } \widetilde{B}(x_{\alpha}^k, c_1\delta ^k) \subset Q_{\alpha}^k \subset \widetilde{B}(x_{\alpha}^k, C_1\delta ^k) =: B(Q_{\alpha}^k).\label{a5}\\
 &4. \text{ If } l \geq k \text{ and } Q_{\beta}^l \subset Q_{\alpha}^k, \text{ then } B(Q_{\beta}^l) \subset B(Q_{\alpha}^k).\label{a6}
 \end{align}
 \hspace{0.5cm} 5. For each $(k, \alpha)$ and each $l \leq k$, there exists a unique $\beta$ such that
 \begin{equation}\label{a3}
 Q_{\alpha}^k \subset  Q_{\beta}^l.
\end{equation}
 \hspace{0.6cm}6. For each $(k, \alpha)$ there exist between 1  and $M$ (a fixed geometric constant) cubes $Q_{\beta}^{k+1}$ such that \\
 \begin{equation}\label{a4}
  Q_{\beta}^{k+1} \subset Q_{\alpha}^k, \text{ and } Q_{\alpha}^k = \bigcup_{Q \in \mathscr{D}_{k+1}, Q\subset Q_{\alpha}^k} Q.
  \end{equation}
The set $Q_{\alpha}^k$ is called a \emph{dyadic cube of generation $k$} with \emph{center point} $x_{\alpha}^k \in Q_{\alpha}^k$ and \emph{side length}~$\delta ^k.$
\end{defn}


\begin{thm}\label{klp}\textbf{(Theorem 2.1 in \cite{KLPW16})}
Let $(X, \rho)$ be a geometrically doubling quasimetric space. Then there exists a system~$\D$ of dyadic cubes with parameters $0 < \delta \leq (12A_0^3)^{-1}$ and $c_1 = (3A_0^2)^{-1}$, $C_1 = 2A_0.$ The construction only depends on some fixed set of countably many center points $x_{\alpha}^k$, satisfying the two inequalities
\begin{equation*}
  \rho (x_{\alpha}^k, x_{\beta}^k) \geq \delta ^k \hspace{0.5 cm} (\alpha \neq \beta),
  \hspace{1.5cm} \min_{\alpha} \rho (x, x_{\alpha}^k) < \delta ^k \text{ for all } x \in X,
\end{equation*}
and a certain partial order $\leq$ among their index pairs $(k, \alpha)$.
\end{thm}

\begin{defn} \label{cubesinX} \cite{KLPW16} \textbf{(Adjacent Systems of Dyadic Cubes)}
In a geometrically doubling quasimetric space $(X, \rho)$, a finite collection $\{\mathscr{D}^t: t = 1,2,\ldots,T\}$ of families $\mathscr{D}^t$ is called a \emph{collection of adjacent systems of dyadic cubes with parameters $\delta \in (0,1)$, $c_1$ and $C_1$} such that $0 < c_1 < C_1 < \infty$ and $C \in [1, \infty)$ if it has the following properties: individually, each $\mathscr{D}^t$ is a system of dyadic cubes with parameters $\delta \in (0,1)$ and $0 < c_1 < C_1 < \infty$; collectively, for each ball $\widetilde{B}(x,r)\subset X$ with $\delta^{k+3} <r< \delta^{k+2}, k \in \mathbb{Z}$, there exist $t \in \{1,2,\ldots,T\}$ and $Q \in \mathscr{D}^t$ of generation $k$ and with center point $^{t}x_{\alpha}^k$ such that $\rho(x, ^tx_{\alpha}^k) < 2A_0\delta ^k$ and
\begin{equation}\label{a8}
  \widetilde{B}(x,r) \subset Q \subset \widetilde{B}(x, Cr).
\end{equation}
\end{defn}

\begin{thm}\label{klp1}\textbf{(Theorem 2.7 in \cite{KLPW16})}
Let $(X, \rho)$ be a geometrically doubling quasimetric space. Then there exists a collection $\{\mathscr{D}^t: t = 1,2,\ldots,T\}$ of adjacent systems of dyadic cubes with parameters $0 < \delta \leq (96A_0^6)^{-1}$ and $c_1 = (12A_0^4)^{-1}$, $C_1 = 4A_0^2$ and $C = 8A_0^3\delta^{-3}$. The center points $^{t}x_{\alpha}^k$ of the cubes $Q \in \mathscr{D}^t_k $ satisfy, for each $t \in \{1,2,\ldots,T\}$, the  two inequalities
\begin{equation*}
 \rho (^tx_{\alpha}^k, ^tx_{\beta}^k) \geq (4A_0^2)^{-1}\delta ^k \hspace{0.5 cm} (\alpha \neq \beta),
  \hspace{1.5cm} \min_{\alpha} \rho (x, ^tx_{\alpha}^k) < 2A_0\delta ^k \text{ for all } x \in X,
\end{equation*}
and a certain partial order $\leq$ among their index pairs $(k, \alpha)$.
\end{thm}

\subsection{Doubling weights vs dyadic doubling weights}\label{subsec:dbl_dydbl_weight}
In this section, we define doubling weights and dyadic doubling weights on metric measure spaces and spaces of homogeneous type.
\begin{defn}
(i) A \emph{weight} on a metric measure space~$(X,d,\mu)$ is a nonnegative locally integrable function $w: X \rightarrow [0,\infty]$.

(ii) A weight~$w$ on a metric measure space~$(X,d,\mu)$ is \emph{doubling} if there is a constant~$C_{\text{dbl}}$ such that for all~$x \in X$ and all~$r>0$,
\begin{equation}\label{eq:dbl_measure_3}
0 < w(\widehat{B}(x,2r)) \leq C_{\text{dbl}} w(\widehat{B}(x,r)) < \infty.
\end{equation}
We recall the notation~$w(E) = \int_E \,w\,d\mu$ where~$E \subset X$.
As in Definition~\ref{defn:dbl_measure}, inequality~\eqref{eq:dbl_measure_3} implies a more general property of the doubling weight~$w$. That is,  for all~$x \in X$, $r > 0$ and~$\lambda > 1$ we have
\begin{equation}\label{eq:dbl_measure_4}
   w(\widehat{B}(x,\lambda r)) \leq C_{\text{dbl}}^{1 + \log_2 \lambda} w(\widehat{B}(x,r)).
\end{equation}

(iii) A weight~$w$ on a metric measure space~$(X,d,\mu)$ equipped with a system~$\mathscr{D}$ of dyadic cubes is \emph{dyadic doubling} if there is a constant~$C_{\text{dydbl}}$ such that for every dyadic cube~$Q \in \mathscr{D}$ and for each child~$Q'$ of~$Q$,
\[0 < w(Q) \leq C_{\text{dydbl}} w(Q') < \infty.\]

 (iv) Similarly, we define weights, doubling weights and dyadic doubling weights on a space of homogeneous type~$(X,\rho,\mu)$ by replacing the ball~$\widehat{B}$ by the quasiball~$\widetilde{B}$.
\end{defn}

It is shown in~\cite{KLPW16} that on a space of homogeneous type~$(X,\rho,\mu)$, if a weight~$w$ is doubling on~$X$ with doubling constant~$C_{\text{dbl}}$, then~$w$ is dyadic doubling with w.r.t.~each of the systems~$\mathscr{D}^t$ of dyadic cubes, $t = 1,\ldots,T$, given by Theorem~\ref{cubesinX}. The dyadic doubling constant can be taken to be~$C_{\text{dydbl}} = C_{\text{dbl}}^N$, with ~$N = 1 + \log_2(2A_0C_1/(c_1\delta))$, where $A_0$ is the quasitriangle constant, and $C_1$, $c_1$ and~$\delta$  are from Theorem~\ref{cubesinX}. The same proof can be applied for a doubling measure~$\mu$ to conclude that~$\mu$ is dyadic doubling with~$C_{\text{dydbl}} = A_1^N$.

\subsection{The function space~$\bmo$ on~$X$}\label{subsec:BMO_X}
In this section, we define the function space~$\bmo$ on metric measure spaces~$(X,d,\mu)$ and on spaces of homogeneous type~$(X,\rho,\mu)$.

\begin{defn}\label{def:BMO_X}
  Let $(X, d, \mu)$ be a metric measure space.
  A locally integrable real-valued function~$f: (X,d,\mu) \rightarrow \R$
  is in~$\bmo(X,d,\mu)$ if
    \begin{equation}\label{eq:BMO_X}
        \|f\|_{\bmo(X,d,\mu)}
        := \sup_{\widehat{B}}\frac{1}{\mu(\widehat{B})}\int_{\widehat{B}} \lf|f(x) - f_{\widehat{B}}\r| \,d\mu(x) < \infty,
    \end{equation}
    where
    $f_{\widehat{B}} := \frac{1}{\mu(\widehat{B})}\int_{\widehat{B}} f(y) \,d\mu(y)$
    is the average of the function~$f$ over the (metric) ball~$\widehat{B} \subset X$.

    Let $\D$ denote any fixed system of dyadic cubes in~$(X,d,\mu)$.
    We define the dyadic $\bmo$ classes $\bmo_{\D}(X,d,\mu)$ as in~\eqref{eq:BMO_X} above, except that the ball~$\widehat{B}$ is replaced by the dyadic cube~$Q \in \D$.
\end{defn}

Let~$(X,\rho,\mu)$ be a space of homogeneous type.
The function classes $\bmo(X,\rho,\mu)$ and $\bmo_{\D}(X,\rho,\mu)$ are defined as in Definition~\ref{def:BMO_X}, except that the ball~$\widehat{B}$ is replaced by the quasiball~$\widetilde{B}$, and the fixed system~$\D$ of dyadic cubes is now in~$(X,\rho,\mu)$.

\subsection{Quasisymmetric maps on~$X$ and their generalised Jacobians $\widehat{J}_f$ and $\widetilde{J}_f$}\label{subsec:quasisymmetic}
The concept of quasisymmetry
is a generalisation of quasiconformality in arbitrary
metric spaces.
We now define $\eta$-\emph{quasisymmetric} maps and their generalised Jacobians.
\begin{defn}\label{def:quasisym}\cite{Tys98} (\textbf{$\eta$-quasisymmetric})
Let~$(X,d_X)$ and~$(Y,d_Y)$ be metric spaces. A homeomorphism $f:(X,d_X) \rightarrow (Y,d_Y)$ is called \emph{$\eta$-quasisymmetric} if there is an increasing homeomorphism $\eta: [0,\infty)\rightarrow [0,\infty)$ so that
\[  \frac{d_X(x,a)}{d_X(x,b)}  \leq \theta \qquad \Rightarrow \qquad \frac{d_Y(f(x),f(a))}{d_Y(f(x),f(b))} \leq \eta(\theta).\]
\end{defn}

Let~$(X,\rho_X)$ and~$(Y,\rho_Y)$ be quasimetric spaces. An $\eta$-quasisymmetric mapping $f:(X,\rho_X) \rightarrow (Y,\rho_Y)$ is defined as in Definition~\ref{def:quasisym} above, except that the metrics~$d_X$
and~$d_Y$ are replaced by the quasimetrics~$\rho_X$ and~$\rho_Y$, respectively.

In Lemma~\ref{lem:homeoequi}, we will show that the $\e$-chain approach preserves the $\eta$-quasisymmetry of functions on~$(X,\rho,\mu)$.

Given a metric measure space~$(X,d,\mu)$,
let~$f$ be an $\eta$-quasisymmetric map from~$(X,d,\mu)$ onto itself.
For each $\mu$-measurable set~$E \subset X$, we define the \emph{pullback measure}~$\mu_f$ by $\mu_f(E) := \mu(f(E))$.
The measure~$\mu_f$ is in fact doubling (see Lemma~\ref{lem:nu_dbl_2}).
We define the \emph{generalised Jacobian of~$f$ w.r.t.~the metric~$d$} by
\begin{equation}\label{eq:JacoQuasi_hat}
    \widehat{J}_f(x) := \lim_{r \rightarrow 0^+}\frac{\mu_f(\overline{B}_{d}(x,r))}{\mu(\overline{B}_{d}(x,r))},
\end{equation}
where $\overline{B}_{d}(x,r) := \{y \in X: d(x,y) \leq r\}$ for all $x \in X$ and $r>0$.


Given a space of homogeneous type~$(X,\rho,\mu)$,  the \emph{generalised Jacobian of~$f$ w.r.t.~the quasimetric~$\rho$}
is defined similarly, except that the ball~$\widehat{B}(x,r)$ is replaced by the quasiball~$\widetilde{B}(x,r)$, and function~$f$ is an $\eta$-quasisymmetric map from~$(X,\rho,\mu)$ onto itself:
\begin{equation}\label{eq:JacoQuasi_tilde}
    \widetilde{J}_f(x) := \lim_{r \rightarrow 0^+}\frac{\mu_f(\overline{B}_{\rho}(x,r))}{\mu(\overline{B}_{\rho}(x,r))},
\end{equation}
where $\overline{B}_{\rho}(x,r) := \{y \in X: \rho(x,y) \leq r\}$ for all $x \in X$ and $r>0$.

Below, when in an already known setting (metric or quasimetric), we will call~$\widehat{J}_f$ and~$\widetilde{J}_f$ the \emph{generalised Jacobian} for short.
Using our generalisation of the Radon--Nikodym Theorem (Theorem~\ref{lem:Nages analog}), we can show that under some additional conditions, the generalised Jacobians~$\widehat{J}_f(x)$ and~$\widetilde{J}_f(x)$  exist and are finite for $\mu$-a.e.~$x \in X$.

The quantities $\widehat{J}_f$ and~$\widetilde{J}_f$ are generalisations of the volume derivative~$V_f$ in~\cite{HK98}, where our doubling measure~$\mu$ has replaced the Hausdorff $Q$-measure, $Q >1$.
Pekka Koskela, one of the authors of~\cite{HK98}, confirmed to us by email that in the definition of $V_f$ in [HK98], the balls should be closed balls, rather than open balls as is stated in [HK98].
This lets us define our generalised Jacobians in terms of closed balls.



\subsection{Weighted inequalities on~$X$}\label{sec:weight}
In this section, we introduce two classes of weight functions, called \emph{$A_p$ weights} and \emph{reverse-H\"{o}lder-$p$ weight}s.


\begin{defn}\cite{Muc72}
  \label{def:Ap}
  (\textbf{$A_p$ weight})
  Let~$(X,d,\mu)$ be a metric measure space.
  Let $\omega(x)$ be a weight on~$X$. For $p$ with $1 < p < \infty$, we
  say $\omega$ is an \emph{$A_p$ weight}, written $\omega\in A_p$ or $\omega\in A_p(X)$,
  if
    \begin{equation}\label{eq2weight}
    A_p(\om)
    := \sup_{\widehat{B}} \left(\intav_{\widehat{B}} \omega\right)
    \left(\intav_{\widehat{B}} \bigg(\frac{1}{\om}\bigg)^{1/(p-1)}\right)^{p-1}
    < \infty.
    \end{equation}
Here the supremum is taken over all
  balls~$\widehat{B}\subset X$. The quantity~$[\om]_{A_p}$ is called
  the \emph{$A_p(X)$~constant} of~$\om$.

Given a system of dyadic cubes $\D$ on $X$ as in Definition~\ref{dyadiccubes}, we define the \emph{dyadic $A_p$} classes $A_p^{\D} = A_p^{\D}(X)$ as in~\eqref{eq2weight}
above except that now the supremum is taken over all
      dyadic cubes $Q\in \D$.
\end{defn}

\begin{defn}
  \label{def:RHp}
  (\textbf{Reverse-H\"{o}lder-$q$ weight})
  Let~$(X,d,\mu)$ be a metric measure space.
  Let $\omega(x)$ be a weight on~$X$. For $q$ with $1 < q < \infty$, we
  say $\omega$ is a \emph{reverse-H\"older-$q$ weight} (reverse-H\"{o}lder, for short), written
  $\omega\in RH_q$ or $\omega\in RH_q(X)$,
  if
    \begin{equation}\label{eq5weight}
    RH_q(\om)
    := \sup_{\widehat{B}} \left(\intav_{\widehat{B}} \omega^q\right)^{1/q}
    \left(\intav_{\widehat{B}} \omega\right)^{-1}
    < \infty.
    \end{equation}
Here the supremum is taken over all
  balls~$\widehat{B}\subset X$. The quantity $[w]_{RH_q}$ is called
  the \emph{$RH_q(X)$~constant} of~$\om$.

 Given a system of dyadic cubes $\D$ on $X$ as in Definition~\ref{dyadiccubes}, we define
  the \emph{dyadic $RH_q$} classes $RH_q^{\D} = RH_q^{\D}(X)$ as in~\eqref{eq5weight}
  above except that now the supremum is taken over all dyadic cubes $Q\in \D$. In addition, one must require explicitly that $\om$ is a dyadic doubling weight. This is a technical requirement which is also present in the Euclidean case.
\end{defn}

Given a space of homogeneous type~$(X,\rho,\mu)$, we define $A_p$ weights, dyadic $A_p$ weights, $RH_q$ weights and dyadic $RH_q$ weights as in Definitions~\ref{def:Ap} and~\ref{def:RHp} above except that the ball~$\widehat{B}(x,r)$ is replaced by the quasiball~$\widetilde{B}(x,r)$.

The definitions of $A_p$ weights and reverse-H\"older-$p$ weights indicate that such a weight cannot degenerate or grow too quickly. This property can be phrased equivalently in terms of how much the logarithm of the weight can  oscillate.

\subsection{Results about measurable functions}\label{sec:meable func}
Given a set~$X$ and a $\sigma$-algebra $\mathfrak{M}$ on~$X$, $(X,\mathfrak{M})$ is called a \emph{measurable space}.
\begin{defn}\label{defn:meable_func}(\textbf{Measurable function})
Let~$(X,\mathfrak{M})$ and~$(Y,\mathfrak{N})$ be measurable spaces. 
A mapping~$h: X \rightarrow Y$ is called~$(\mathfrak{M}, \mathfrak{N})$-\emph{measurable}, or just \emph{measurable} when~$\mathfrak{M}$ and~$\mathfrak{N}$ are understood, if~$h^{-1}(E) \in \mathfrak{M}$ for all~$E \in \mathfrak{N}$.
\end{defn}
 More details about measurable functions can be found in~\cite[Chapter~2]{Fol99}.
Below we collect some properties related to measurable functions.
These will be applied for the generalised Jacobian~$\widehat{J}_f$ in Section~\ref{sec:R1_mec}.
\begin{prop}[\textbf{Results about measurable functions}]
\label{prop:properties_meable_func}
  Let~$h: X \rightarrow [0,\infty]$ be a measurable function on a
  measurable space $(X,\mathfrak{M})$. Then\\
  \indent\textup{(i)} for every~$n \in \R^+$ the function~$h^{n}(x)$ is also measurable, and\\
  \indent \textup{(ii)}  the reciprocal~$1/h(x)$ is measurable, except on the set~$\{x \in X: h(x) = 0\}$.\\
  Suppose~$\mu_1$ and~$\mu_2$ are two measures such that for all measurable sets~$E \subset X$ we have~$\mu_1(E) \sim_C \mu_2(E)$, where~$C>0$ is a constant. Then\\
  \indent\textup{(iii)}$\int_X h \,d\mu_1 \sim_C \int_X h \,d\mu_2.$\\
  Suppose~$\mathcal{B}_X$ is the Borel $\sigma$-algebra generated by the collection of open sets in~$X$. Assign to the set~$X$ a metric~$d$ and a doubling measure~$\mu$, which is defined on~$\mathcal{B}_X$. For each~$x \in X$ and~$r>0$, set~$B(x,r) := \{y \in X: d(x,y)<r\}$. Then\\
  \indent\textup{(iv)} for each fixed~$r$, the function~$\varphi(x) = \varphi_r(x):= \mu(\overline{B}(x,r))$ is measurable. Here for each~$x \in X$ and~$r>0$, let $\overline{B}(x,r) := \{y \in X: d(x,y) \leq r\}$ denote the closed ball on~$X$.
\end{prop}

\begin{proof}
  Properties~(i) and~(ii) are straightforward from~\cite[Proposition~2.3]{Fol99}.
  For property~(iii), first consider characteristic functions~$h(x) = \chi_E(x)$,
  then simple functions~$h(x) = \sum_{i = 1}^{n} c_i \chi_{E_i}(x)$, where~$c_i \geq 0$ and~$E_i \subset X$,
  then arbitrary nonnegative measurable functions~$h$.

Now we will show property~(iv).
Note that the function~$\varphi$ takes~$X$ to~$[0,\infty)$.
By Proposition~2.3 in~\cite{Fol99}, to show~$\varphi$ is a measurable function, it suffices to show that for all~$a > 0$
\[\varphi^{-1}([0,a)) = \{x \in X: \mu(\overline{B}(x,r)) < a\} \in \mathcal{B}_X.\]
Thus, it suffices to show that~$\varphi^{-1}([0,a))$ is a Borel set in~$X$. In particular, it is enough to show that~$\varphi^{-1}([0,a))$ is open in~$X$.

Fix~$a>0$.
Fix~$x \in \varphi^{-1}([0,a))$.
For~$\e > 0$, define a neighbourhood of~$x$ by
$$N_{x,\e} := \{x' \in X: d(x,x') < \e\}.$$
We claim that there exists~$\e > 0$ such that with~$r^* := r + \e$ we have
$\mu(\overline{B}(x,r))  \leq \mu(B(x,r^*))< a.$
This will be shown at the end of this proof.
Then for such an~$\e$, take~$x' \in N_{x,\e}$  and~$y \in \overline{B}(x',r)$. By the triangle inequality we have
$$d(y,x) \leq d(y,x') + d(x',x) < r + \e = r^*.$$
Thus, $\overline{B}(x',r) \subset B(x,r^*)$, and so $\mu(\overline{B}(x',r)) \leq \mu(B(x,r^*)) < a$. Consequently, $x' \in \varphi^{-1}([0,a))$.
Since this is true for all~$x' \in N_{x,\e}$, we have~$N_{x,\e} \subset \varphi^{-1}([0,a))$.
Since this is true for all~$x \in \varphi^{-1}([0,a))$, we see that~$\varphi^{-1}([0,a))$ is open in~$X$.
Since this is true for all~$a > 0$, we conclude that~$\varphi$ is a measurable function.

We are left with proving our claim.
Fix~$a>0$.
Fix~$x \in \varphi^{-1}([0,a))$.
We recall a result in~\cite[Exercise~15, p.~52]{Fol99}: if~$\{f_n\}$ is a sequence of measurable functions from~$X$ to~$[0,\infty]$, $f_n$ decreases pointwise to~$f$, and~$\int f_1 < \infty$, then we have $\int f = \lim \int f_n$.
We will apply this result for~$f = \chi_{\overline{B}(x,r)}$ and~$f_n = \chi_{B(x,r+1/n)},$ where~$n \in \N$.
Notice that for each~$n$, $f_n$ is a characteristic function from~$X$ to~$[0,\infty]$, so it is measurable.
As~$\mu$ is a doubling measure, for each~$n \in \N$ we obtain
\[\int_X f_n \,d\mu = \int_X \chi_{B(x,r+1/n)} \,d\mu = \mu(B(x,r+1/n)) < \infty.\]
Also, $f_n$ decreases pointwise to~$f$. To see this, consider~$y \notin \overline{B}(x,r)$. Then~$f(y) =~0$ and $f_n(y) = \chi_{B(x,r+1/n)}(y) \rightarrow 0$ as~$n \rightarrow \infty$, as for~$n$ sufficiently large, $d(y, \overline{B}(x,r)) > 1/n >0$. If~$y \in \overline{B}(x,r)$, then for each~$n \in \N$, we have~$y \in B(x,r+1/n)$, because~$\overline{B}(x,r) \subset B(x,r+1/n)$. Thus
$f(y) = \chi_{\overline{B}(x,r)}(y) =  \chi_{B(x,r+1/n)}(y) = 1.$
Therefore,  we can conclude that
\[\mu(\overline{B}(x,r)) = \int_X \chi_{\overline{B}(x,r)} \,d\mu =
\lim_{n \rightarrow \infty} \int_X  \chi_{B(x,r+1/n)} \,d\mu
=  \lim_{n \rightarrow \infty} \mu(B(x, r + 1/n)).\]
Hence, we may choose~$n$ sufficiently large that
\[\mu(\overline{B}(x,r)) \leq \mu(B(x, r + 1/n))
< \frac{\mu(B(x,r))+a}{2} < a.\]
Setting~$\e = 1/n$, our claim is established.
\end{proof}

\subsection{Results about $A_{\infty}$-related measures}\label{sec:A_infty measure}
\begin{defn}\label{prop:properties_A_infty}\textbf{($A_{\infty}$-related)}
On a metric space~$(X,d)$, a measure~$\mu_1$ is \emph{$A_{\infty}$-related }to a measure~$\mu_2$ if for each~$\lambda > 0$, there exists~$\delta > 0$ such that
  \[\mu_2(E) < \delta \mu_2(B) \quad \Rightarrow \quad \mu_1(E) < \lambda\mu_1(B),\]
  whenever~$E$ is a measurable subset of a ball~$B$.
\end{defn}
 Below we collect some properties of $A_{\infty}$-related measures.
 They will be applied for measures~$\mu$, $\mathcal{H}_{\al}$, $\mu_f$, $\sigma_f$ in Section~\ref{sec:R1_mec}.
\begin{prop}\textbf{(Results about $A_{\infty}$-related measures)}
  Let~$(X,d)$ be a metric space.
  Let~$\mu_1$ and~$\mu_2$ be measures on~$X$.
  Then the following statements hold. \\
  \indent\textup{(a)} If~$\mu_1$ is comparable to~$\mu_2$,
  then~$\mu_1$ is $A_{\infty}$-related to~$\mu_2$.\\
  \indent\textup{(b)} If  $\mu_1$ is $A_{\infty}$-related to~$\mu_2$ and
  $\mu_2$ is $A_{\infty}$-related to~$\mu_3$,
  then~$\mu_1$ is $A_{\infty}$-related to~$\mu_3$.
\end{prop}
\begin{proof}
  Property~(a) is straightforward from the comparability of~$\mu_1$ and~$\mu_2$. Property~(b) follows from the definition of $A_{\infty}$-relatedness.
\end{proof}

\section{Proof of Theorem~\ref{thm.logBMO}}\label{sec:logBMO}
In this section, we prove our first main result, namely Theorem~\ref{thm.logBMO}.
The five steps of the proof of Theorem~\ref{thm.logBMO} are outlined in the Introduction. Our remaining task is proving the theorems mentioned there.
In Section~\ref{subsec:CZcubes}, we establish a Calder\'{o}n--Zygmund decomposition stated in terms of dyadic cubes.
In Section~\ref{subsec:RHr_charac}, we establish two properties of the dyadic reverse-H\"{o}lder weights.
In Section~\ref{subsec:RHr_imply_dRHr}, we prove that a reverse-H\"{o}lder weight is also a dyadic reverse-H\"{o}lder weight.
In Section~\ref{subsec:RhripliesAp}, we show that a dyadic reverse-H\"{o}lder weight is also a dyadic $A_p$ weight.
In Section~\ref{subsec:logAp_BMO}, we show that the logarithm of a dyadic $A_p$ weight is in dyadic~$\bmo$.

The setting of these sections is in spaces of homogeneous type~$(X,\rho,\mu)$.
We believe that the results presented in these sections are of independent interest, beyond our use of them in the proof of Theorem~\ref{thm.logBMO}.

To simplify the notation, in this section only,  when we say~$X$, we mean~$(X,\rho,\mu)$. When we say the ball $B(x,r)$, we mean the quasiball~$\widetilde{B}(x,r)$.
On the space~$X$, we can generate a collection~$\{\D^t: t = 1, \ldots, T\}$ of adjacent systems of dyadic cubes of~$X$, as in Definition~\ref{cubesinX} and Theorem~\ref{klp1}. When we talk about a fixed dyadic grid~$\D$ of cubes or a system~$\D$ of dyadic cubes, we mean a system~$\D^t$, when $t \in \{1,\ldots, T\}$ is fixed.

\subsection{Calder\'{o}n--Zygmund decomposition of~$(X,\rho,\mu)$ with cubes}\label{subsec:CZcubes}

   In this section, we start by establishing a Calder\'{o}n--Zygmund decomposition on spaces of homogeneous type $(X,\rho,\mu)$. In fact, this result still holds if the measure~$\mu$ is just dyadic doubling and not necessary doubling. Recall that an analogous Calder\'{o}n--Zygmund decomposition on $(X,\rho,\mu)$ has been derived previously in~\cite{CW71}. However, that version is in terms of balls, and it does not give us property~(i) in Theorem~\ref{CZ}, which is the main property that we use in proofs of other results.
   Here we derive two other analogs of the Calder\'{o}n--Zygmund decomposition in terms of dyadic cubes.
   Theorem~\ref{CZ} is called the \emph{local version} as the decomposition takes place entirely in a cube~$Q_0$. This is also the version that is used in the proof of our first main result (Theorem~\ref{thm.logBMO}). Theorem~\ref{CZ2} is called the \emph{global version}. We include it here because we believe that it has its own interest.
\begin{thm}\label{CZ} \textbf{(Calder\'{o}n--Zygmund decomposition on $(X,\rho,\mu)$: local version)}
  Given a space of homogeneous type $(X, \rho, \mu)$,
  let $\D$ denote any fixed system of dyadic cubes in~$X$.
  Take $f \in L^1(X)$ with $\supp f \subset Q_0$, where~$Q_0 \in \D$.
  Define~$\al_0 := \frac{1}{\mu(Q_0)}\int_{Q_0} f \,d\mu$.
  Define the dyadic maximal function $M$  w.r.t.~$\D$ by
  \begin{equation}\label{eq:CZ1}
  M f(x) := \sup_{\substack{Q \ni x \\ Q \in \D, Q \subset Q_0}}\intq |f(y)| \,d\mu(y),
  \end{equation}
  where the supremum is taken over all dyadic cubes $Q \in \D$ containing $x$ and included in~$Q_0$. Let $\al > \al_0$ and
  $\Om_{\al} := \{x \in Q_0: M f(x) > \al\}.$ 
  Then $\Om_{\al}$ can be written as a disjoint union of dyadic cubes $\{Q_j\}$ with the following three properties.\\
  \indent\textup{(i)} For each cube $Q_j$,
  \[\al < \intqj |f(x)| \,d\mu(x) \leq A_1^N\al, \]
  where $A_1 \geq 1$ is the doubling constant of $\mu$ and~$N := 1 + \log_2(2A_0C_1/(c_1\delta))$. \\
  \indent\textup{(ii)} For $\mu$-a.e. $x \in X\backslash\bigcup_j Q_j$, we have~$M f(x) \leq \al$. \\
  \indent\textup{(iii)} $\mu(\Om_{\al}) \leq \frac{1}{\al}\int_X |f(x)| \,d\mu(x)$.
\end{thm}

\begin{thm}\label{CZ2} \textbf{(Calder\'{o}n--Zygmund decomposition on $(X, \rho, \mu)$: global version)}
  Given a space of homogeneous type $(X, \rho, \mu)$ such that~$\mu(X) < \infty$,
  let $\D$ denote any fixed system of dyadic cubes in~$X$.
  Take $f \in L^1(X)$.
  Define the dyadic maximal function $M$  w.r.t.~$\D$ by
  \begin{equation*}
  M f(x) := \sup_{\substack{Q \ni x \\ Q \in \D}}\intq |f(y)| \,d\mu(y)
  \end{equation*}
  where the supremum is taken over all dyadic cubes $Q \in \D$ containing $x$. Let $\al > 0$ be such that
  $\Om_{\al} := \{x \in X: M f(x) > \al\} \text{ has finite measure}.$
  Then $\Om_{\al}$ can be written as a disjoint union of dyadic cubes $\{Q_j\}$ with the following three properties.\\
  \indent\textup{(i)} For each cube $Q_j$,
  \[\al < \intqj |f(x)| \,d\mu(x) \leq A_1^N\al, \]
  where $A_1 \geq 1$ is the doubling constant of $\mu$ and~$N := 1 + \log_2(2A_0C_1/(c_1\delta))$. \\
  \indent\textup{(ii)} For $\mu$-a.e. $x \in X\backslash\bigcup_j Q_j$, we have~$M f(x) \leq \al$. \\
  \indent\textup{(iii)} $\mu(\Om_{\al}) \leq \frac{1}{\al}\int_X |f(x)| \,d\mu(x)$.
\end{thm}
The conclusion of both theorems is the same, but their hypotheses are slightly different. In particular, in Theorem~\ref{CZ}, the supremum in the definition of~$Mf(x)$ is only taken over all dyadic cubes containing~$x$ and included in~$Q_0$, and $\al > \al_0$, where~$\al_0 := \frac{1}{\mu(Q_0)}\int_{Q_0} f \,d\mu$.
By contrast, in Theorem~\ref{CZ2},  the supremum in the definition of~$Mf(x)$ is taken over all dyadic cubes containing~$x$, and $\al > 0$; also, $\Om_{\al}$ is assumed to have finite measure.
Now we are going to prove Theorem~\ref{CZ}.

\medskip
\noindent\emph{Proof of Theorem~\ref{CZ}.}
The proof follows the proof of the (global) Calder\'{o}n--Zygmund decomposition of $f \in L^1(\R^n)$ given in~\cite[Lemma~1, Section IV.3]{Ste93}, noting the following points.

In our (local) setting, we see immediately that for each~$x \in \Omega_{\al}$ there is a maximal dyadic cube containing~$x$ and contained in~$Q_0$, since by definition of~$Mf(x)$ there is a cube~$Q_1 \in x$, $Q_1 \subset Q_0$, with~$\intav_{Q_1} |f(y)| \,d\mu(y) > \al > \al_0 := \intav_{Q_0} |f(y)| \,d\mu(y)$, and there are only finitely many cubes containing~$Q_1$ and contained in~$Q_0$.

We need not explicitly assume that~$\Om_{\al}$ has finite measure. Indeed, for~$\al > \al_0$, $\mu(\Om_{\al}) \leq \mu(Q_0) < \infty$.

By property~\eqref{a2} of the dyadic cubes in~$X$, any two dyadic cubes are nested or disjoint.

In the second inequality in property (i), we obtain~$A_1^N\al$, not~$2^n\al$ as in the Euclidean~$\R^n$ case. The reason for the difference is that for the parent~$\widetilde{Q}_j$ of a cube~$Q_j$ in~$X$, we have~$\mu(\widetilde{Q}_j) \leq A_1^N\mu(Q_j)$. For as mentioned in Section~\ref{subsec:dbl_dydbl_weight}, since~$\mu$ is a doubling measure, it is also dyadic doubling with the dyadic doubling constant~$C_{\text{dydbl}} = A_1^N$ where~$N = 1 + \log_2[2A_0C_1/(c_1\delta)]$.

In property~(ii), the conclusion that~$Mf(x) \leq \al$ is straightforward from the definition of~$\Om_{\al}$. In the Euclidean case, since~$|f(x)| \leq Mf(x)$ by the Lebesgue Differentiation Theorem, one can obtains that~$|f(x)| \leq \al$. However, the Lebesgue Differentiation Theorem may or may not hold in a given space of homogeneous type~$(X,\rho,\mu)$. See also Remark~\ref{rem:x_notin_UQ} for a variant of our Theorem~\ref{CZ} which yields an upper bound for~$|f(x)|$, not just for~$Mf(x)$.
 \hfill\(\Box\)

\medskip
\noindent\emph{Proof of Theorem~\ref{CZ2}.}
The proof of Theorem~\ref{CZ2} is the same as that of Theorem~\ref{CZ}, except the part showing the existence of the maximal dyadic cubes.
For each~$x \in \Omega_{\al}$, we can show that there exists a maximal dyadic cube containing~$x$ by contradiction, using property~\eqref{a5} of dyadic cubes, the fact that~$\mu$ is a doubling measure, and the assumption that~$\mu(\Omega_{\al}) < \infty$.  \hfill\(\Box\)

\medskip
Below we state three remarks related to Theorem~\ref{CZ}. These remarks are also apply to Theorem~\ref{CZ2}.
\begin{rem}\label{rem:x_notin_UQ}
  In Theorem~\ref{CZ},
   if we impose an extra assumption, use a slightly different definition of the dyadic maximal function~$Mf(x)$, and sacrifice other conclusions, then we can obtain a different version of conclusion~(ii). In particular, in addition to the hypotheses of Theorem~\ref{CZ}, we assume that the Lebesgue Differentiation Theorem holds in~$X$.
   In the definition of~$Mf(x)$, the supremum is taken not only over all cubes in a fixed dyadic grid~$\D$ containing~$x$ and included in~$Q_0$, but also over all cubes in a collection~$\{\D^t\}:= \{\D^t : t = 1,\ldots,T\}$ of dyadic grids containing~$x$ and included in~$Q_0$, given by Theorem~\ref{klp1}. Specifically, $Mf(x)$ is now defined as
  \[M f(x) := \sup_{\substack{  Q \in \{\D^t\} \\ Q \ni x,  Q \subset Q_0}}\intq |f(y)| \,d\mu(y).\]
  Then we obtain a collection $\{Q_j: Q_j \in \bigcup_{t =1}^{T}\D^t\}$ of dyadic cubes such that $\Omega_{\al} = \bigcup_j Q_j$. The conclusions that we have to sacrifice are the disjointness of the cubes~$Q_j$, as they are not necessary disjoint; and conclusion~(iii), as it relies on the disjointness of~$Q_j$.

  A different version of conclusion~(ii) in Theorem~\ref{CZ} that we gain is that~$|f(x)| \leq A_1^m\al$, where $m = 1 +\log_2 (8A_0^3\delta^{-3})$.

  This remark is not used in the proofs of our main results, but it is of interest on its own. We omit the proof.
\end{rem}


\begin{rem}
  If $\al_1 > \al_2$, then
$\{x \in X: M f(x) > \al_1\} \subset \{x \in X: M f(x) > \al_2\}$,
and by the maximality of the cubes, each dyadic cube in the decomposition at level $\al_1$ is contained in a dyadic cube in the decomposition at level $\al_2$.
\end{rem}

\begin{rem}\label{rem:CZ_ddbl}
  Given a dyadic doubling weight~$w$, Theorem~\ref{CZ} still holds if we replace~$d\mu$ by~$w\, d\mu$, and $\mu(Q)$ by~$w(Q) = \int_Q \,w\,d\mu$, where~$Q \in \D$. In that case, the constant~$A_1^N$ appearing in property~(i) is replaced by the dyadic doubling constant~$C_{\text{dydbl}}$ of the weight~$w$.
\end{rem}


\subsection{Properties of~$RH_q^{\D}$}\label{subsec:RHr_charac}
In Theorem~\ref{thm2} below, we establish two properties of the class $RH_q^{\D}(X)$  of dyadic reverse-H\"{o}lder-$r$ weights.
\begin{thm}\label{thm2}
   Given a space of homogeneous type $(X, \rho, \mu)$,
  let $\D$ denote any fixed system of dyadic cubes.
  Suppose $w$ is a weight on~$X$ and $w \in RH_q^{\D}(X)$ with the~$RH_q^{\D}(X)$ constant~$[w]_{RH_q^{\D}}$ for some~$q \in (1,\infty)$.
  Then there exists $\e\in (0, \infty)$   such that for all dyadic cubes $Q \in \D$  and all $\mu$-measurable subsets $E$ of $Q$ we have
  \begin{equation}\label{22.1}
    \frac{w(E)}{w(Q)} \leq [w]_{RH_q^{\D}}\bigg(\frac{\mu(E)}{\mu(Q)}\bigg)^{\e}.
  \end{equation}
  Furthermore, there exist $\gamma, \lambda \in (0,1)$ such that
  \begin{equation}\label{23.1}
     w(E) < \gamma w(Q) \quad \Rightarrow \quad \mu(E) < \lambda\mu(Q).
  \end{equation}
\end{thm}

\begin{proof}
Theorem~\ref{thm2} is a generalisation from the Euclidean setting of the implications $\text{(c)} \Rightarrow \text{(d)} \Rightarrow \text{(e)}$ between parts~(c), (d) and~(e) in~\cite[Theorem 9.3.3]{Gra09}. The proof given there works perfectly on spaces of homogeneous type~$(X,\rho,\mu)$.
\end{proof}

\subsection{A reverse-H\"{o}lder weight is also a dyadic reverse-H\"{o}lder weight}\label{subsec:RHr_imply_dRHr}
%
Lemma~\ref{lem:RHr_implies_dRHr} says that if a weight~$w$ is a reverse-H\"{o}lder weight, then $w$ is also a dyadic reverse-H\"{o}lder weight. In other words, if~$w$  has the reverse-H\"{o}lder-$q$ property w.r.t.~balls in~$X$, then~$w$ also has the reverse-H\"{o}lder-$q$ property w.r.t.~each of the systems~$\D^t, t \in \{1, \ldots,T\}$, of dyadic cubes on~$X$.
\begin{lem}\label{lem:RHr_implies_dRHr}
  Suppose $X$ is a space of homogeneous type $(X, \rho, \mu)$ equipped with the systems~$\D^t$ of dyadic cubes, $t \in \{1, \ldots,T\}$, given by Theorem~\ref{klp1}.
  Suppose~$w$ is a weight on~$X$ and $w \in RH_q(X)$ with the $RH_q(X)$ constant~$[w]_{RH_q}$, where~$q \in (1,\infty)$. Then for the same~$q$, $w \in RH_q^{\D^t}(X)$, w.r.t.~each of the systems~$\D^t$ of dyadic cubes.
  The~$RH_q^{\D^t}(X)$ constant of~$w$ is $[w]_{RH_q^{\D^t}} := [w]_{RH_q} A_1^{m/q}C_{\text{dbl}}^m$,
  where~$A_1$ and~$C_{\text{dbl}}$ are the doubling constants of~$\mu$ and~$w$, respectively, and $m = 1 + \log_2\frac{C_1}{c_1}$ with $c_1$ and $C_1$ as in property~\eqref{a5}.
\end{lem}
\begin{proof}
  Let~$\D$ denote any of the systems~$\D^t$.
  Fix a dyadic cube $Q  \in \mathscr{D}$ of generation $k \in \mathbb{Z}$, centred at $z$.
Let~$B_2 := B(z, C_1\delta ^k)$ and~$ m := 1 + \log_2(C_1/c_1)$, where~$c_1$ and~$C_1$ are from property~\eqref{a5} of dyadic cubes.
By properties~\eqref{a5} of dyadic cubes and~\eqref{eq:dbl_measure_2} of doubling measures, together with the facts that $w \geq 0$ for all $x \in X$, $w \in RH_q(X)$ with constant~$[w]_{RH_q}$ and $w$ is doubling with constant~$C_{\text{dbl}}$, we obtain
\begin{eqnarray*}
  \bigg(\frac{1}{\mu(Q)}\int_{Q} w^q \,d\mu\bigg)^{1/q} &\leq&
   \bigg(\frac{A_1^m }{\mu(B_2)}\int_{B_2} w^q \,d\mu\bigg)^{1/q}\\
  &\leq&  [w]_{RH_q} A_1^{m/q}\frac{1 }{\mu(B_2)}\int_{B_2} w \,d\mu \\
  &\leq& [w]_{RH_q} A_1^{m/q}C_{\text{dbl}}^m \frac{1 }{\mu(Q)} \int_{Q} w \,d\mu \\
  &=:& [w]_{RH_q^{\D^t}} \intav_Q w \,d\mu,
\end{eqnarray*}
where $[w]_{RH_q^{\D^t}} := [w]_{RH_q} A_1^{m/q}C_{\text{dbl}}^m$.
Thus $w \in RH_q^{\D^t}(X)$.
This completes the proof of Lemma~\ref{lem:RHr_implies_dRHr}.
\end{proof}
\subsection{A dyadic reverse-H\"{o}lder weight is also a dyadic $A_p$ weight}\label{subsec:RhripliesAp}
\begin{thm}\label{thm3}
  Given a space of homogeneous type $(X, \rho, \mu)$, such that~$\mu$ is Borel regular,
  let $\D$ denote any fixed system of dyadic cubes in~$X$.
  Suppose $w$ is a weight on~$X$ and $w \in RH_q^{\D}(X)$ for some~$q \in (1,\infty)$.
  Then $w \in A_p^{\D}(X)$ for some~$p \in (1,\infty)$. 
\end{thm}

\begin{proof}
Theorem~\ref{thm3} is a generalisation to~$(X,\rho,\mu)$ of its analogue in the Euclidean setting; this Euclidean analogue is established during the proof of Theorem~3 in~\cite[Section~5.1]{Ste93}. In the original proof, the two main ingredients are the Calder\'{o}n--Zygmund decomposition of~$f \in L^1(\R^n)$ given in~\cite[Lemma~1, Section~IV.3]{Ste93}, and the property of~$w \in RH_q(\R^n)$ given in~\cite[Theorem~9.3.3.(e)]{Gra09}.
We have generalised both of these ingredients to the setting of spaces of homogeneous type~$(X,\rho,\mu)$: see Theorems~\ref{CZ} and~\ref{thm2} above.

Following the structure of the original proof in~\cite{Ste93}, to show that~$w \in A_p^{\D}(X)$ for some~$p \in (1,\infty)$, it suffices to show that there exist some~$c >0$ and~$\bar{q} >1$ such that for each cube~$Q_0 \in \D$ we have
 \begin{equation}\label{14.1}
   \bigg(\frac{1}{w(Q_0)}\int_{Q_0} w^{1-\bar{q}} \,d\mu \bigg)^{1/\bar{q}} \leq c\frac{\mu(Q_0)}{w(Q_0)},
 \end{equation}
where as usual~$w(Q) = \int_Q w \,d\mu$.
Note that in~\cite{Ste93}, the cube~$Q_0$ is normalised such that~$\mu(Q_0) = w(Q_0)=1$, which leads to~$\al_0 := \mu(Q_0)/ w(Q_0) = 1$. However, the proof works without this normalisation.
To make the calculations more explicit, we work with a general (non-normalised) dyadic cube~$Q_0 \in \D$.

Fix a cube~$Q_0 \in \D$.
Let~$f = w^{-1}\chi_{Q_0}$.
We will apply our (local) Calder\'{o}n--Zygmund decomposition (Theorem~\ref{CZ}) to the \emph{dyadic maximal function with weight $w$}, defined by
  \[M_w f(x) := \sup_{\substack{Q \ni x \\ Q \in \D, Q \subset Q_0}}\frac{1}{w(Q)}\int_Q |f(y)|w(y) \,d\mu(y)
  =:  \sup_{\substack{Q \ni x \\ Q \in \D, Q \subset Q_0}}\frac{1}{w(Q)}\int_Q |f(y)|\,d\mu_2(y),\]
  where the supremum is taken over all dyadic cubes in~$\D$ containing $x$ and contained in~$Q_0$.
Note that in~\cite{Ste93}, the supremum is taken over all dyadic cubes in~$\D$ containing $x$, without requiring~$Q \subset Q_0$.

Notice that the weighted maximal function~$M_w f$ is the same as the unweighted maximal function~$M f$ defined in~\eqref{eq:CZ1} but with the measure~$\mu_2$ in place of~$\mu$.
Moreover, since~$w$ is a dyadic doubling weight, as noted in Remark~\ref{rem:CZ_ddbl}, Theorem~\ref{CZ} also holds for~$M_w f$. The only difference is that the constant~$A_1^N$ appearing in property~(ii) is replaced by the dyadic doubling constant~$C_{\text{dydbl}} > 1$ of~$w$.

Let $\al_s = C_{\text{dydbl}}^{Ms}\al_0$, where~$M, s \in \N$. Note that  $\al_s$ is the substitute for~$2^{Ms}$ in~\cite{Ste93}.
Define the set $E^s := \{x \in Q_0: M_w f(x) > \al_s\}$.
Again, following the proof in~\cite{Ste93}, we can show that
\begin{equation}\label{14.2}
\mu(E^s) < \lambda^s \mu(Q_0),
\end{equation}
where here~$\lambda \in (0,1)$ is from Theorem~\ref{thm2}.

Now we are ready to prove~\eqref{14.1}.
We note that since~$\mu$ is Borel-regular, it is Borel semiregular. As noted in~\cite{AM15}, it follows that the Lebesgue Differentiation Theorem holds in~$X$.
So we have
\begin{eqnarray}\label{eq:lhs1}
  \frac{1}{w(Q_0)}\int_{Q_0} w^{1-\bar{q}} \,d\mu
  = \frac{1}{w(Q_0)}\int_{Q_0} f^{\bar{q}-1} \,wd\mu
  \leq \frac{1}{w(Q_0)}\int_{Q_0} [M_w f(x)]^{\bar{q}-1} \,d\mu.
\end{eqnarray}
The integral~\eqref{eq:lhs1} can be broken into
\[\underbrace{\frac{1}{w(Q_0)}\int_{Q_0\backslash E^0} [M_w f(x)]^{\bar{q}-1} \,d\mu}_{\textup{(I)}}
+ \underbrace{\frac{1}{w(Q_0)} \sum_{s=0}^{\infty} \int_{E^s \backslash E^{s+1}} [M_w f(x)]^{\bar{q}-1} \,d\mu}_{\textup{(II)}}. \]
Note that because of our slightly different definition of~$M_wf$, the sets over which the integrals~(I) and~(II) are evaluated are slightly simpler than those in~\cite{Ste93}.

The integral~(I) is majorised by~$(\mu(Q_0)/ w(Q_0))^{\bar{q}}$.
Using~\eqref{14.2} we can show
\[\text{(II)} < \bigg(\frac{\mu(Q_0)}{w(Q_0)}\bigg)^{\bar{q}} \sum_{s=1}^{\infty}  C_{\text{dydbl}}^{M(s+1)(\bar{q}-1)}\lambda^s.\]
Since $\lambda < 1$, the geometric series $\sum_{s=0}^{\infty} C_{\text{dydbl}}^{M(s+1)(\bar{q}-1)}\lambda^s$ converges if $\bar{q}$ is sufficiently close to~1, specifically, if $\bar{q} < \log \lambda^{-1} / (M \log C_{\text{dydbl}}) + 1$. We have therefore proved~\eqref{14.1}. In turn, this shows that $w \in A_p^{\D}$ for~$p = \bar{q}/(\bar{q}-1)$ for each~$\bar{q}$ in this range, completing the proof of Theorem~\ref{thm3}.
\end{proof}


\subsection{The logarithm of a dyadic $A_p$ weight is in $\bmo_{\mathscr{D}}$}\label{subsec:logAp_BMO}
In Theorem~\ref{AplogW} below, we show that the logarithm of an $A_p^{\mathscr{D}}(X)$ weight is in $\bmo_{\mathscr{D}}(X)$.
This result is motivated by its analogues on Euclidean spaces~$\R^n$.
See for example \cite[Exercise~9.2.3]{Gra09} and~\cite[Lemma 2]{PWX11}.
\begin{thm}\label{AplogW}
 Given a space of homogeneous type $(X, \rho, \mu)$,
 let $\D$ denote any fixed system of dyadic cubes in~$X$.
 Suppose $w$ is a weight on $X$ and~$w \in A_p^{\mathscr{D}}(X)$ for some $p \in (1,\infty)$.
 Then $\log w(x) \in \bmo_{\mathscr{D}}(X)$ with
 \[\|\log w\|_{\bmo_{\mathscr{D}}} \leq [w]_{A_p^{\mathscr{D}}} + (p-1)[w]_{A_p^{\mathscr{D}}}^{1/(p-1)}.\]
\end{thm}
To establish Theorem~\ref{AplogW}, we need to use the following lemma
about the oscillation of the logarithm of an~$A_p$ weight.
\begin{lem}\label{lem:osc_log_Ap}
 Given a space of homogeneous type $(X, \rho, \mu)$,
 let $\D$ denote any fixed system of dyadic cubes in~$X$.
 Suppose $w$ is a weight on $X$ and~$w \in A_p^{\mathscr{D}}(X)$ for some $p \in (1,\infty)$.
 Let~$\lambda(x) = \log w(x)$. Then
 \begin{equation*}
 \intav_{Q}e^{\lambda(x)-\lambda_Q} \,d\mu
 \leq [w]_{A_p^{\D}} \quad \text{ and } \quad
 \intav_{Q}e^{\frac{\lambda_Q-\lambda(x)}{p-1}} \,d\mu
 \leq [w]_{A_p^{\D}}^{1/(p-1)}.
\end{equation*}
\end{lem}
The proofs of Theorem~\ref{AplogW} and Lemma~\ref{lem:osc_log_Ap} are straightforward and proceed as in the Euclidean case. We omit their proofs.
\section{Further Results on Spaces of Homogeneous Type~$(X,\rho,\mu)$}
\label{sec:quasi_result}
In this section, we present further results on the setting of spaces of homogeneous type, which are necessary for the later sections, as well as having their own interest.
In Sections~\ref{subsec:mu_f dbl} and~\ref{subsec:m_f Borel regular}, we show that the measure induced by an $\eta$-quasisymmetric map is doubling and Borel regular, respectively.
In Section~\ref{subsec:Vitali covering thm}, we generalise the Vitali Covering Theorem.
In Section~\ref{subsec:generalise_Radon_Nikodym_thm}, we establish a generalisation of the Radon--Nikodym Theorem.

To simplify the notation, in this section only,  when we say~$X$, we mean~$(X,\rho,\mu)$. When we say the ball $B(x,r)$, we mean the quasiball~$\widetilde{B}(x,r)$.
Note that any results proved on spaces of homogeneous type~$(X, \rho, \mu)$ also hold on metric measure spaces~$(X,d,\mu)$.

\subsection{The measure induced by a quasisymmetric map is doubling}\label{subsec:mu_f dbl}
Given a $\mu$-measurable set~$E \subset X$ and an $\eta$-quasisymmetric map from~$X$ onto itself, we recall the pullback measure~$\mu_f$ by~$\mu_f(E) := \mu(f(E))$.
In Lemma~\ref{lem:nu_dbl_2} below, we will prove that~$\mu_f$ is doubling, under the extra assumption that $X$ has nonempty $\tau$-annuli (Definition~\ref{Xspace}).
\begin{lem}\label{lem:nu_dbl_2}
Let~$(X,\rho,\mu)$ be a space of homogeneous type that has nonempty $\tau$-annuli for some~$\tau \in (0,1)$.
Suppose $f: X \rightarrow X$ is an $\eta$-quasisymmetric map from~$X$ onto itself.
Then~$\mu_f$ is a doubling measure with doubling constant~$C_{\mu_f}$ depending on~$A_1, \tau$ and~$\eta$.
\end{lem}

\begin{rem}\label{rem:maasalo}
We note that a result analogous to Lemma~\ref{lem:nu_dbl_2} is stated in~\cite[Proposition~4.7]{Maa06} in the setting where~$X$ is a $Q$-regular metric measure space with~$Q > 1$ that is doubling and rectifiably connected, and~$\mu$ is the Hausdorff $Q$-measure. The proof of Proposition~4.7 in~\cite{Maa06} relies on Proposition~4.6 in~\cite{Maa06}, which in turns relies on Proposition~4.5 in~\cite{Maa06}.
However, there is a point in the proof of Proposition~4.6 that seems unclear, as we explain below.

The technique that we use here to prove Lemma~\ref{lem:nu_dbl_2} is completely independent of that used in~\cite{Maa06}.

Let us recall Propositions~4.5 and~4.6 in~\cite{Maa06}, stated here as Propositions~\ref{prop:maa4.5} and~\ref{prop:maa4.6}, respectively.

\begin{prop}[\textbf{Proposition~4.5 in~\cite{Maa06}}]\label{prop:maa4.5}
  Let~$X$ and~$Y$ be metric spaces. If $f: X \rightarrow Y$ is $\eta$-quasisymmetric and if $A_1 \subset A_2 \subset X$ are such that $0 < \dia A_1 < \dia A_2 < \infty$, then $\dia f(A_2)$ is finite and
  \[
  \frac{1}{2\eta\Big( \frac{\dia A_2}{\dia A_1}\Big)}
  \leq \frac{\dia f(A_1)}{\dia f(A_2)}
  \leq \eta \bigg( \frac{2\dia A_1}{\dia A_2}\bigg).
  \]
\end{prop}

\begin{prop}[\textbf{Proposition~4.6 in~\cite{Maa06}}]\label{prop:maa4.6}
   Let~$X$ and~$Y$ be metric spaces. Let $f: X \rightarrow Y$ be $\eta$-quasisymmetric. Then for all~$x \in X$ and $r > 0$ there exist two constants $0 < r_x < R_x $ such that
   \[B(f(x),r_x) \subset f(B(x,r)) \subset B(f(x), R_x).\]
\end{prop}

In the proof of Proposition~\ref{prop:maa4.6}, the author fixes a ball $B := B(x,r) \in X$ with~$r >0$. Next, the author lets $r_x > 0$ be arbitrary for the moment, and lets the function $f: X \rightarrow Y$ be $\eta$-quasisymmetric, so that $f^{-1}: Y \rightarrow X$ is $\eta'$-quasisymmetric, where $\eta'(t) = 1/\eta^{-1}(t^{-1})$ for~$t>0$.
Then the author applies Proposition~\ref{prop:maa4.5} to the quasisymmetric map~$f^{-1}$ and the sets $A_1 = B(f(x),r_x) \subset Y$ and~$A_2 = B \subset X$.
Notice that in order to apply Proposition~\ref{prop:maa4.5}, the sets~$A_1$ and~$A_2$ must be in the same space. Namely, we must have~$A_1 \subset A_2 \subset Y$. However, the sets~$B(f(x),r_x)$ and~$B$ are in different spaces. 
We think there may be a typo here and that the set~$A_2$ should be~$f(B) \subset Y$ instead of~$B$. Even so, we still do not see how to guarantee that $B(f(x),r_x) \subset f(B)$.
\end{rem}


Below we introduce Proposition~\ref{prop:nu dbl} and Lemma~\ref{lem:nu_dbling}, which will be used to prove Lemma~\ref{lem:nu_dbl_2}.
\begin{prop}\label{prop:nu dbl}
Under the same conditions as in Lemma~\ref{lem:nu_dbl_2}
  choose~$k \geq 1/\tau$ and
  fix~$x \in X$, $a >0$, $b>0$.
  For all~$y_a \in B(x,ka)\backslash B(x,a)$ and~$y_b \in B(x,kb)\backslash B(x,b)$ we have
  \begin{equation}\label{eq:claim_nu dbl}
    \frac{\rho(f(x), f(y_b))}{\rho(f(x),f(y_a))} \leq \eta\bigg(\frac{kb}{a}\bigg).
  \end{equation}
\end{prop}
Note that the conclusion of Proposition~\ref{prop:nu dbl} still holds under the weaker assumptions~$y_a \in X\backslash B(x,a)$ and~$y_b \in B(x,kb)$.
\begin{proof}
  Since~$\tau \in (0,1)$, $k \geq 1/\tau > 1$.
  The existence of points~$y_a$ and~$y_b$ is because~$X$ has nonempty $\tau$-annuli and because of the way~$k$ is chosen.
  Inequality~\eqref{eq:claim_nu dbl} is straightforward from the $\eta$-quasisymmetry of~$f$.
\end{proof}

 The technical Lemma~\ref{lem:nu_dbling} below gives some control over the distortion of sets under an $\eta$-quasisymmetric map.
\begin{lem}\label{lem:nu_dbling} \textbf{(Distortion Lemma)}
Let~$(X,\rho,\mu)$ be a space of homogeneous type.
Suppose $f: X \rightarrow X$ is an $\eta$-quasisymmetric map from~$X$ onto itself.
Choose~$\theta$ and~$k$ such that
\[0 \leq \eta(\theta) \leq \frac{1}{3}
\quad \text{ and } \quad
k \geq \frac{1}{\theta}.\]
For each ball~$B(x,r)$ in~$X$, let
\[s := \sup_{x' \in f(B(x,r))} \rho(f(x),x')
\quad \text{ and } \quad
t := \inf_{x' \in X \backslash f(B(x,kr))} \rho(f(x),x').\]
Then~$s < t$ and hence
\begin{equation}\label{eq0:nu_dbl}
B(f(x),s) \subset B(f(x), t).
\end{equation}
\end{lem}
Notice that under the conditions of Lemma~\ref{lem:nu_dbling}, there is a concentric annulus centred at~$f(x)$ that separates~$f(B(x,r))$ and~$ X \backslash f(B(x,kr))$.
We also note that when we apply Lemma~\ref{lem:nu_dbling} in the proof of Lemma~\ref{lem:nu_dbl_2} below, we will also assume that the space~$(X,\rho,\mu)$ has nonempty $\tau$-annuli and~$k \geq 1/\tau$, where~$\tau \in (0,1)$. However, these two extra assumptions are not needed for the proof of Lemma~\ref{lem:nu_dbling}.

The proof of Lemma~\ref{lem:nu_dbling} is presented at the end of this section. Now we will use Lemma~\ref{lem:nu_dbling} to prove Lemma~\ref{lem:nu_dbl_2}. We first establish properties~\eqref{eq15:nu_dbl} and~\eqref{eq21:nu_dbl} below, then use them to show that $\mu_f$ is doubling, meaning there exists $C_{\mu_f} >1$ such that $\mu_f(B(x, 2r')) \leq C_{\mu_f}\mu_f(B(x,r'))$ for all~$x \in X$ and~$r' > 0$.

\medskip
\noindent \emph{Proof of Lemma~\ref{lem:nu_dbl_2}.}
Recall that our~$(X,\rho,\mu)$ has nonempty $\tau$-annuli for some~$\tau \in (0,1)$.
Choose~$\theta$ and~$k$ such that
\[0 \leq \eta(\theta) \leq \frac{1}{3}
\quad \text{ and } \quad
k \geq \max\bigg(\frac{1}{\theta}, \frac{1}{\tau}\bigg).\]

Fix a ball~$B(x,r)$ where~$x \in X$ and~$r >0$.
Let~$y \in B(x,r) \backslash B(x,r/k)$ and $z \in B(x,kr)\backslash B(x,r)$.
Applying Proposition~\ref{prop:nu dbl} with~$a = r/k$, $b = r$, $y_a = y$ and~$y_b=z$ we have
\[\frac{\rho(f(x),f(z))}{\rho(f(x),f(y))} \leq \eta \bigg(\frac{kr}{r/k} \bigg) = \eta(k^2).\]
As in Lemma~\ref{lem:nu_dbling}, define
\[s := \sup_{x' \in f(B(x,r))} \rho(f(x),x')
\quad \text{ and } \quad
t := \inf_{x' \in X \backslash f(B(x,kr))} \rho(f(x),x').\]
So
\[\frac{\rho(f(x),f(z))}{\eta(k^2)} \leq \rho(f(x),f(y))
\leq \sup_{x' \in f(B(x,r))} \rho(f(x),x') = s.\]
This together with the results in Lemma~\ref{lem:nu_dbling} and the way~$t$ is defined give us
\begin{equation}\label{eq15:nu_dbl}
B\bigg(f(x), \frac{\rho(f(x),f(z))}{\eta(k^2)}\bigg) \subset B(f(x),s) \subset B(f(x),t) \subset f(B(x,kr)).
\end{equation}

Next take $y_1 \in B(x,2k^3r) \backslash B(x,2k^2r)$. Again applying Proposition~\ref{prop:nu dbl}, this time with~$a = r$, $b = 2k^2r$, $y_a= z$ and~$y_b = y_1$ we have
\begin{equation}\label{eq17:nu_dbl}
    \frac{\rho(f(x),f(y_1))}{\rho(f(x),f(z))} \leq \eta \bigg(\frac{2k^3r}{r} \bigg) = \eta(2k^3).
\end{equation}

 Let~$r_1 := 2kr$. Set
 \[s_1 := \sup_{x' \in f(B(x,2kr))} \rho(f(x),x') = \sup_{x' \in f(B(x,r_1))} \rho(f(x),x'), \quad \text{ and}\]
 \[t_1 := \inf_{x' \in X \backslash f(B(x,2k^2r))} \rho(f(x),x') = \inf_{x' \in X \backslash f(B(x,kr_1))} \rho(f(x),x').\]
 Notice that with our choices of~$\theta$, $k$, $r_1$, $s_1$ and~$t_1$, we may apply Lemma~\ref{lem:nu_dbling} to conclude that~$s_1 < t_1$.
Moreover, since~$f$ is a homeomorphism, $ y_1 \notin B(x,2k^2r)$ implies that $ f(y_1) \notin f(B(x,2k^2r))$. Thus
\begin{equation}\label{eq19:nu_dbl}
    t_1 \leq \rho(f(x),f(y_1)).
\end{equation}
By the definition of~$s_1$, Lemma~\ref{lem:nu_dbling} and inequalities~\eqref{eq19:nu_dbl} and~\eqref{eq17:nu_dbl} we have
\begin{equation}\label{eq21:nu_dbl}
f(B(x,2kr)) \subset B(f(x),s_1) \subset B\big(f(x),\eta(2k^3)\rho(f(x),f(z))\big).
\end{equation}

Finally, using~\eqref{eq21:nu_dbl}, the doubling property of the measure~$\mu$, and~\eqref{eq15:nu_dbl}, we show that the measure~$\mu_f$ is doubling, with doubling constant~$C_{\mu_f}$ depending on~$A_1, \tau$ and~$\eta$. Consider 
\begin{eqnarray}\label{eq22:nu_dbl}
 \mu\big(f(B(x,2kr))\big) &\leq& \mu\big(B\big(f(x),\eta(2k^3)\rho(f(x),f(z))\big)\big) \noz \\
   &\leq& A_1^{1 + \log_2(\eta(2k^3)\eta(k^2))}
   \mu\bigg(B\Big(f(x),\frac{\rho(f(x),f(z))}{\eta(k^2)}\Big)\bigg) \noz \\
   &\leq& A_1^{1 + \log_2(\eta(2k^3)\eta(k^2))}\mu\big(f(B(x,kr))\big).
\end{eqnarray}
Now for each~$r' > 0$, let~$r := r'/k$. Then~$kr = r'$.
By~\eqref{eq22:nu_dbl} we have
\[\mu_f(B(x,2r')) \leq A_1^{1 + \log_2(\eta(2k^3)\eta(k^2))} \mu_f(B(x,r')).\]
Hence $\mu_f$ is doubling with the doubling constant $C_{\mu_f} := A_1^{1 + \log_2(\eta(2k^3)\eta(k^2))}, $
where~$A_1$ is the doubling constant of~$\mu$.  \qedhere \hfill \(\Box\)


 To complete this section, we give the proof of Lemma~\ref{lem:nu_dbling}.

\medskip
\noindent \emph{Proof of Lemma~\ref{lem:nu_dbling}}.
Let~$\theta$ and~$k$ be chosen as in Lemma~\ref{lem:nu_dbling}.
Fix a ball~$B(x,r)$ in~$X$.
By the definition of~$s$ and~$t$, there exist sequences~$\{u_n\} \subset f(B(x,r))$ and $\{v_n\} \subset X \backslash f(B(x,kr))$ such that
\[\lim_{n \rightarrow \infty} \rho(f(x),u_n) = s \quad \text{ and } \quad
\lim_{n \rightarrow \infty} \rho(f(x),v_n) = t.\]
Fix~$\e > 0$.
Then there exists~$M \in \N$ such that for all~$n \geq M$, we have
\begin{equation}\label{eq5:nu_dbl}
    s - \e < \rho(f(x),u_n) \leq s  \quad \text{ and } \quad   t  \leq \rho(f(x),v_n) < t + \e.
\end{equation}
Taking $\e = s/2$, inequalities~\eqref{eq5:nu_dbl} gives
\begin{equation*}
\frac{s}{2} < \rho(f(x),u_n) \leq s \quad \text{ and } \quad
t \leq \rho(f(x), v_n) < \frac{2t+s}{2}.
\end{equation*}
This yields immediately
\begin{equation}\label{eq7:nu_dbl}
    \frac{s}{2t+s} < \frac{\rho(f(x),u_n)}{\rho(f(x), v_n)} \leq \frac{s}{t}.
\end{equation}
Moreover, since~$u_n \in f(B(x,r))$, $v_n \in X\backslash f(B(x,kr))$ and~$f$ is a homeomorphism, we have
$\rho(x,f^{-1}(u_n)) < r$ and $\rho(x,f^{-1}(v_n)) \geq kr$.
From this together with the fact that~$1/k \leq \theta$, we obtain
\begin{equation}\label{eq1.2:nu_dbl}
    \frac{\rho(x,f^{-1}(u_n))}{\rho(x,f^{-1}(v_n))} < \frac{r}{kr} = \frac{1}{k} \leq \theta.
\end{equation}
Now using inequality~\eqref{eq7:nu_dbl}, the quasisymmetry of~$f$ and the fact that~$\eta(\theta) \leq 1/3$  we have
\[\frac{s}{2t+s} <
\frac{\rho(f(x),u_n)}{\rho(f(x),v_n)} \leq \eta(\theta) \leq \frac{1}{3},\]
which is equivalent to~$s < t$.
Again,~\eqref{eq0:nu_dbl} follows immediately. \hfill\(\Box\)

\subsection{The measure induced by a quasisymmetric map is Borel regular}\label{subsec:m_f Borel regular}
Recall that we are assuming the measure~$\mu$ is defined on a $\sigma$-algebra $\mathcal{M}$ which contains all Borel sets and all quasiballs in~$X$.
Let~$\mathcal{B}_X$ be the Borel $\sigma$-algebra generated by the collection~$\Oo$ of open sets in~$X$.
Hence, $\mathcal{B}_X \subset \mathcal{M}$, so~$\mu$ is defined on~$\mathcal{B_X}$, and so~$\mu$ is a Borel measure.

We recall the definition of the pullback measure~$\mu_f(E) := \mu(f(E))$ for all $\mu$-measurable set~$E \in X$ and an $\eta$-quasisymmetric map~$f$ from~$X$ onto itself. Since~$f$ is a homeomorphism, it follows immediately that~$\mu_f$ is a measure.

In this section, we start by showing that the measure~$\mu_f$ is also a Borel measure, if~$\mu$ is a Borel measure (Lemma~\ref{lem:mu_f regular}).
Then we prove that~$\mu_f$ is also Borel regular, if~$\mu$ is Borel regular (Lemma~\ref{lem:m_f Borel regular}).
This result is used in Section~\ref{sec:R1_mec} to show the existence of~$\widehat{J}_f$.

\begin{lem}\label{lem:mu_f regular}
Suppose~$(X,\rho,\mu)$ is a space of homogeneous type with $\mu$ being a Borel measure.
Let $f$ be an $\eta$-quasisymmetric map from $(X,\rho,\mu)$ onto itself.
Then the measure~$\mu_f$ is a Borel measure.
\end{lem}
\begin{proof}
%

To show that~$\mu_f$ is a Borel measure, we have to show that~$\mu_f$ is defined on every Borel set~$E \subset X$.
Since $\mu_f(E) = \mu(f(E))$ and $\mu$ is a Borel measure, it is sufficient to show that for each  Borel set~$E$, $f(E)$ is also a Borel set. In other words, the homeomorphism~$f$ preserves the collection of Borel sets. To prove this, we first establish the following claim.
\begin{claim}\label{claim:f(B) is B}
Suppose~$(X,\rho,\mu)$ is a space of homogeneous type.
Let~$\Oo$ be the collection of open sets in~$X$.
Let~$\mathcal{B}_X$ be the $\sigma$-algebra generated by~$\Oo$. That is,~$\mathcal{B}_X$ is the Borel $\sigma$-algebra on~$X$.  Let~$f: X \rightarrow X$ be a homeomorphism of~$X$ onto itself.
Define $f(\mathcal{B}_X) := \{f(E): E \in \mathcal{B}_X\}.$
Then~$f(\mathcal{B}_X) = \mathcal{B}_X$.
\end{claim}
\noindent \emph{Proof of Claim~\ref{claim:f(B) is B}.}
Since~$\mathcal{B}_X$ is a $\sigma$-algebra and~$f$ is a homeomorphism of~$X$ onto itself, it is straightforward to show that~$f(\mathcal{B}_X)$ is closed under countable unions and complements, and so~$f(\mathcal{B}_X)$ is a $\sigma$-algebra.
Moreover, it follows immediately from the continuity of~$f$ that~$\Oo \subset f(\mathcal{B}_X)$.
Since by definition $\mathcal{B}_X$ is the smallest $\sigma$-algebra that contains~$\Oo$, we have~$\mathcal{B}_X \subset f(\mathcal{B}_X)$.
Applying the same argument to~$f^{-1}$, we obtain~$\mathcal{B}_X \subset f^{-1}(\mathcal{B}_X)$, and so~$f(\mathcal{B}_X) \subset f \circ f^{-1}(\mathcal{B}_X) = \mathcal{B}_X$.
 \hfill \(\Box\)

From Claim~\ref{claim:f(B) is B}, we can see that~$f(E) \in \mathcal{B}_X$ for all Borel sets~$E \in \mathcal{B}_X$, which completes the proof of Lemma~\ref{lem:mu_f regular}.
\end{proof}
If we assume further that~$\mu$ is Borel regular, then~$\mu_f$ is also Borel regular.
\begin{lem}\label{lem:m_f Borel regular}
Suppose~$(X,\rho,\mu)$ is a space of homogeneous type with $\mu$ a Borel regular measure.
Let $f$ be an $\eta$-quasisymmetric map from $(X,\rho,\mu)$ onto itself.
Then the measure~$\mu_f$ is Borel regular.
\end{lem}
\begin{proof}
Since~$f$ and~$f^{-1}$ are homeomorphisms of~$X$ onto itself, the collection of closed subsets of~$X$ is preserved by~$f$ and~$f^{-1}$, and so for all~$E \in \mathcal{B}_X$ we have
\begin{eqnarray}\label{eq1a:mu_f_regular}
  \mu_f(E) &=& \mu(f(E)) \noz\\
   &=&  \sup \{\mu(F): F \text{ closed}, F \subset f(E)\} \noz\\
   &=& \sup \{\mu(f(V)): f(V) \text{ closed}, f(V) \subset f(E)\} \noz\\
   &=& \sup \{\mu_f(V): V \text{ closed}, V \subset E\}.
\end{eqnarray}

Second, we must show that for all~$E \in \mathcal{B}_X$,
\begin{equation*}
\mu_f(E) = \inf \{\mu_f(U): U \text{ open}, E \subset U\}.
\end{equation*}
This follows by an analogous argument to that for~\eqref{eq1a:mu_f_regular}, applied to the open sets~$U$ containing~$E$.

Thus~$\mu_f$ is Borel regular, as required.
\end{proof}


\subsection{Vitali Covering Theorem on~$(X,\rho,\mu)$}\label{subsec:Vitali covering thm}
We develop the basic covering theorem and the Vitali Covering Theorem on spaces of homogeneous type~$(X,\rho,\mu)$. They are presented in Theorems~\ref{thm:basic cover} and~\ref{thm:vitali in X} below.
Theorem~\ref{thm:basic cover} is used to prove Theorem~\ref{thm:vitali in X}.
Theorem~\ref{thm:vitali in X} will be used in the next section to establish the Radon--Nikodym Theorem on~$(X,\rho,\mu)$.
The proof of these covering theorems follows similar ideas to those in the proofs of Theorems~1.6 and~1.2 in~\cite{Hei01}, which are special cases of these covering theorems in the setting of metric measure spaces. We start with the basic covering theorem.

\begin{thm}\label{thm:basic cover}\textbf{(Basic covering theorem on~$(X,\rho,\mu)$)}
Let~$(X,\rho,\mu)$ be a space of homogeneous type. 
Let~$F$ be a family of balls in~$X$ of uniformly bounded radius.
Then there exists a subfamily~$G$ of~$F$ such that
\begin{equation}\label{eq1:basic_cover}
\bigcup_{B \in F} B \subset \bigcup_{B \in G} CB, \text{ where } C = A_0 + 4A_0^2.
\end{equation}
In fact, every ball~$B_F$ from~$F$ meets a ball~$B_G$ from~$G$ with radius at least half that of~$B_F$. Specifically, for each ball~$B_F \in F$, there exists a ball~$B_G \in G$ such that
\[B_F \cap B_G \neq \emptyset \quad \text{ and }\quad r(B_G) \geq \frac{1}{2}r(B_F).\]
\end{thm}

\begin{proof}
  Theorem~\ref{thm:basic cover} is a generalisation from metric measure spaces~$(X,d,\mu)$ to spaces of homogeneous type~$(X,\rho,\mu)$ of Theorem~1.2 in~\cite{Hei01}. The proof given there goes through almost unchanged. The only difference is the use of the quasitriangle inequality instead of the triangle inequality to obtain~\eqref{eq1:basic_cover}. As the reader may notice, the constant~$C = A_0 + 4A_0^2$ in~\eqref{eq1:basic_cover} is the substitute for~$C=5$ in~\cite{Hei01}.
  As usual~$A_0$ is the constant appearing in the quasitriangle inequality for~$(X,\rho,\mu)$.
\end{proof}

Next, the basic covering theorem allows us to prove the Vitali Covering Theorem.
\begin{thm}\label{thm:vitali in X}\textbf{(Vitali Covering Theorem on~$(X,\rho,\mu)$)}
Let~$(X,\rho,\mu)$ be a space of homogeneous type. 
Let~$A$ be a bounded subset of~$X$.
Let~$F$ be a collection of closed balls centred at~$A$ such that the balls have uniformly bounded radii and
\[\inf_{B(a,r) \in F} r = 0 \quad \text{ for each } a \in A.\]
Then there is a countable disjoint subfamily~$G$ of~$F$ such that the balls in~$G$ cover $\mu$-almost all of~$A$, namely
\[\mu(A \backslash \bigcup_{G} B) = 0.\]
\end{thm}

\begin{proof}
Theorem~\ref{thm:vitali in X} is a generalisation to spaces of homogeneous type~$(X,\rho,\mu)$ of its analogue in the setting of metric measure spaces~$(X,d,\mu)$, given in~\cite[Theorem~1.6]{Hei01}. The main ingredient of the original proof is the basic covering theorem given in~\cite[Theorem~1.2]{Hei01}, which we have generalised to~$(X,\rho,\mu)$: see Theorem~\ref{thm:basic cover} above. Once we have Theorem~\ref{thm:basic cover} in hand, the proof of Theorem~\ref{thm:vitali in X} can be carried out as in the original proof, with the constant~$C = A_0 + 4A_0^2$ in place of~$C=5$.
\end{proof}

\subsection{Radon--Nikodym Theorem on~$(X,\rho,\mu)$}\label{subsec:generalise_Radon_Nikodym_thm}
In this section, we discuss an analogue of the well known Radon--Nikodym Theorem, which is stated in Theorem~\ref{lem:Nages analog} below.
This result will be used in Sections~\ref{sec:R1_mec} and~\ref{sec:R1_quasimec} to establish the existence of the generalised Jacobians $\widehat{J}_f$ and~$\widetilde{J}_f$.
\begin{thm}\label{lem:Nages analog}\textbf{(Radon--Nikodym Theorem on~$(X,\rho,\mu)$)}
Let $(X,\rho,\mu)$ be a space of homogeneous type such that the measure~$\mu$ is Borel regular.
Suppose~$\nu$ is another Borel regular measure on~$X$ that is absolutely continuous w.r.t.~the given measure~$\mu$.
 For each~$x \in X$ and~$r > 0$, define the closed ball~$\overline{B}(x,r)$ by
$\overline{B}(x,r) := \{y \in X: \rho(x,y) \leq r\}.$
Then the Radon--Nikodym derivative of~$\nu$ w.r.t.~$\mu$,
$$D(\nu,\mu,x):= \lim_{r \rightarrow 0^+ } \frac{\nu(\overline{B}(x,r))}{\mu(\overline{B}(x,r))},$$
exists and is finite for $\mu$-almost every ($\mu$-a.e.)~$x \in X$. Furthermore, for each $\mu$-measurable set~$S \subset X$,
\[\nu(S) = \int_S D(\nu,\mu,x) \,d\mu(x).\]
\end{thm}
Given two Borel regular measures~$\mu$ and~$\nu$, the measure~$\nu$ is said to be \emph{absolutely continuous w.r.t.}~$\mu$ if for all Borel sets~$E \subset X$, $\mu(E) = 0$ implies~$\nu(E) = 0$.

The proof of Theorem~\ref{lem:Nages analog} requires the use of the Vitali Covering Theorem in~$(X,\rho,\mu)$ in Theorem~\ref{thm:vitali in X} above.
\begin{proof}
Theorem~\ref{lem:Nages analog} is motivated by its analogues on Euclidean spaces~$\R^n$ and on metric measure spaces~$(X,d,\mu)$; see~\cite[Theorem~2.12]{Mat95} and~\cite[Lemma~A.0.7]{Sha99}, respectively.
The key ingredient in the proofs given there is the Vitali Covering Theorem, which we have generalised to spaces of homogeneous type~$(X,\rho,\mu)$; see Theorem~\ref{thm:vitali in X} above.
Once the Vitali Covering Theorem is available, Theorem~\ref{lem:Nages analog} can be proved following the same argument as in~\cite{Mat95} and~\cite{Sha99}.

Note that in the statements of~\cite[Theorem~2.12]{Mat95} and~\cite[Lemma~A.0.7]{Sha99} the balls are not explicitly described as being closed, though within their proofs it is clear that these balls are assumed to be closed.
We have chosen to state this assumption explicitly in Theorem~\ref{lem:Nages analog}, to bring out that it is essential in order to apply the Vitali Covering Theorem.
\end{proof}


\section{Proof of Theorem~\ref{thm.R1.metric}}\label{sec:R1_mec}
In  this section, we establish our second main result, namely a generalisation of Reimann's Theorem~1 to metric measure spaces~$(X,d,\mu)$, stated in Theorem~\ref{thm.R1.metric}. 

The four main steps to prove Theorem~\ref{thm.R1.metric} are outlined in the Introduction. Sections~\ref{subsec:Jfhat_exist}--\ref{subsec:log_Jfhat_BMO} correspond to Steps~1--4 of the proof.
In Section~\ref{subsec:Jfhat_exist}, we show that the generalised Jacobian~$\widehat{J}_f$ exists and is finite for $\mu$-a.e. $x \in X$. In Section~\ref{subsec:Jfhat_meable}, we prove that~$\widehat{J}_f$ is measurable. This is required to show that~$\widehat{J}_f$ is a reverse-H\"{o}lder weight in Section~\ref{subsec:Jfhat_RHr}. Lastly, in Section~\ref{subsec:log_Jfhat_BMO}, by applying Theorem~\ref{thm.logBMO}, we conclude that $\log \widehat{J}_f \in \bmo(X)$.

In this section only, when we write~$X$, we mean~$(X,d,\mu)$, and when we write~$B$, we mean the (metric) ball~$\widehat{B}$.


\subsection{Existence of~$\widehat{J}_f$}\label{subsec:Jfhat_exist}
To show the existence of~$\widehat{J}_f$, we will apply the Radon--Nikodym Theorem (Theorem~\ref{lem:Nages analog}) for the measure~$\nu = \mu_f$. Recall the measure~$\mu_f$ is defined by $\mu_f(E) := \mu(f(E))$, where~$E \subset X$ is $\mu$-measurable and~$f$ is an $\eta$-quasisymmetric map of~$X$ onto itself.
To do so, it is required that the measure~$\mu_f$ is Borel regular and absolutely continuous w.r.t.~the measure~$\mu$.
The first property is shown in Lemma~\ref{lem:m_f Borel regular}. The second property is shown in Lemma~\ref{lem:mu_f abs cts} below.
\begin{lem}\label{lem:mu_f abs cts}
Under the same conditions as in Theorem~\ref{thm.R1.metric}, the measure~$\mu_f$ is absolutely continuous w.r.t.~$\mu$.
\end{lem}
\begin{proof}
Fix a Borel set~$E \in X$ with~$\mu(E) = 0$.
As noted in Remark and Convention~3.4 in \cite{HK98} and shown in Lemma~C.3 in~\cite{Sem96}, the Hausdorff $\al$-measure~$\mathcal{H}_{\al}$ is comparable to~$\mu$. Thus~$\mathcal{H}_{\al}(E) = 0$.
Additionally, by Corollary~7.13 in~\cite{HK98}, the measure~$\mathcal{H}_{\al}(f(\cdot))$ is absolutely continuous w.r.t.~the Hausdorff $\al$-measure~$\mathcal{H}_{\al}(\cdot)$, where~$f$ is an $\eta$-quasisymmetric map. This implies~$\mathcal{H}_{\al}(f(E)) = 0$.
Since this is true for all Borel sets~$E \subset X$ with~$\mu(E) = 0$, we conclude that~$\mu_f$ is absolutely continuous w.r.t.~$\mu$.
\end{proof}

\begin{lem}\label{lem:J_f_hat exist}
  Under the same conditions as in Theorem~\ref{thm.R1.metric}, the generalised Jacobian~$\widehat{J}_f(x)$ exists and is finite for $\mu$-a.e.~$x \in X$.
\end{lem}
\begin{proof}
For each~$x \in X$ and~$r >0$,  denote the closed ball in~$(X,d,\mu)$ by
$$\overline{B}(x,r) := \{y \in X: d(x,y) \leq r\}.$$
Notice that for all~$x \in X$ the generalised Jacobian~$\widehat{J}_f$ coincides with the Radon--Nikodym derivative of~$\mu_f$ w.r.t.~$\mu$:
$$\widehat{D}(\mu_f,\mu,x):= \lim_{r \rightarrow 0^+ } \frac{\mu_f(\overline{B}(x,r))}{\mu(\overline{B}(x,r))}
= \widehat{J}_f(x). $$
By Theorem~\ref{lem:Nages analog}, $\widehat{D}(\mu_f,\mu,x)$
exists and is finite $\mu$-a.e.~$x \in X$.
This completes the proof.
\end{proof}

\subsection{$\widehat{J}_f$ is measurable}\label{subsec:Jfhat_meable}
We recall the definition of a measurable function in Definition~\ref{defn:meable_func}.
We also note that the measure~$\mu$ associated with the metric measure space~$(X,d,\mu)$ in Theorem~\ref{thm.R1.metric} is defined on a $\sigma$-algebra $\mathcal{M}$ which contains all Borel sets and quasiballs in~$X$.
In this section, we will show that under the same conditions as in Theorem~\ref{thm.R1.metric}, the generalised Jacobian determinant~$\widehat{J}_f$ is measurable.
In fact,~$\widehat{J}_f \in L^+(X)$, where
  \[L^+(X) := \{g: X \rightarrow [0,\infty] \text{ and $g$ is measurable}\}.\]

\begin{lem}\label{lem:continuous_lim}
Suppose~$\{r_j\}$ is a sequence such that $r_j > 0$  for each~$j \in \N$, and $\lim_{j \rightarrow \infty} r_j = 0$.
Under the same conditions as in Theorem~\ref{thm.R1.metric}, for~$\mu$-a.e. $x \in X$ we have
   \begin{equation}\label{eq_lem:continuous_lim}
   \widehat{J}_f(x) = \lim_{r \rightarrow 0^+}\frac{\mu_f(\overline{B}(x,r))}{\mu(\overline{B}(x,r))}
   = \lim_{j \rightarrow \infty} \frac{\mu_f(\overline{B}(x,r_j))}{\mu(\overline{B}(x,r_j))}.
   \end{equation}
\end{lem}
\begin{proof}
We recall the standard result that given a function~$h: \R \rightarrow \R$ and~$p \in \R$, then $\lim_{z \rightarrow p} h(z) = L$ if and only if for all sequences~$\{a_j\}$ such that~$\lim_{j \rightarrow \infty} a_j = p$, we have $\lim_{j \rightarrow \infty} h(a_j) = L$~\cite[Theorem~4.2]{Rud76}.

Notice that for each fixed~$x \in X$, the Jacobian determinant~$\widehat{J}_f(x)$ is actually the limit of a function~$h(r)$ as~$r \rightarrow 0^+$:
\[ \widehat{J}_f(x) = \lim_{r \rightarrow 0^+}\frac{\mu_f(\overline{B}(x,r))}{\mu(\overline{B}(x,r))} =: \lim_{r \rightarrow 0^+} h(r). \]
By Lemma~\ref{lem:J_f_hat exist}, this limit exists and is finite for $\mu$-a.e. $x \in X$.
Consider a sequence~$\{r_j\}$ such that for each~$j \in \N$, $r_j > 0$ and~$\lim_{j \rightarrow \infty} r_j = 0$. Then by the standard result mentioned above,
equation~\eqref{eq_lem:continuous_lim} holds for~$\{r_j\}$.
In particular, the limit~$\lim_{j \rightarrow \infty} h(r_j)$
also exists and equals $\widehat{J}_f(x)$.
\end{proof}

\begin{lem}\label{lem:Jfhat measurable}
  Under the same conditions as in Theorem~\ref{thm.R1.metric}, the Jacobian determinant~$\widehat{J}_f \in L^+(X)$. That is,~$\widehat{J}_f$ is a measurable function from~$X$ to~$[0,\infty)$.
\end{lem}
\begin{proof}
It is clear that~$\widehat{J}_f(x) \geq 0$ for $\mu$-a.e.~$x \in X$.
Let~$\{r_j\}$ be a sequence of radii such that $r_j > 0$ for each~$j \in \N$,  and~$\lim_{j \rightarrow \infty} r_j = 0$.
For each~$j \in \N$, define
\[g_j(x) := \frac{\mu_f(\overline{B}(x,r_j))}{\mu(\overline{B}(x,r_j))}.\]
We claim that for each~$j \in \N$, $g_j$ is measurable.
Since~$\mu$ and~$\mu_f$ are both doubling measures, by Proposition~\ref{prop:properties_meable_func}(iv), the mappings~$\varphi_j(x) := \mu(\overline{B}(x,r_j))$ and~~$\psi_j(x) := \mu_f(\overline{B}(x,r_j))$ are measurable.
Consequently, using Proposition~\ref{prop:properties_meable_func}(ii) and Proposition~2.6 in~\cite{Fol99}, which says that the product of measurable functions is measurable, we conclude that for each~$j$, $g_j(x) = \mu_f(\overline{B}(x,r_j))/\mu(\overline{B}(x,r_j))$ is measurable.

By Lemma~\ref{lem:continuous_lim}, for $\mu$-a.e.~$x \in X$ we have
 \[ \widehat{J}_f(x)
   = \lim_{j \rightarrow \infty} \frac{\mu_f(\overline{B}(x,r_j))}{\mu(\overline{B}(x,r_j))}
   = \lim_{j \rightarrow \infty} g_j(x). \]
As shown in~\cite[Proposition~2.7]{Fol99}, the limit of a sequence of measurable functions is measurable. Thus we conclude that~$\widehat{J}_f$ is measurable.
\end{proof}

\subsection{$\widehat{J}_f$ is a reverse-H\"{o}lder weight}\label{subsec:Jfhat_RHr}
Now we will show that the generalised Jacobian~$\widehat{J}_f$ is a reverse-H\"{o}lder weight.
Our proof relies on a result in~\cite[Theorem~7.11]{HK98}, which shows that the volume derivative~$V_f \in RH_q(X)$, where~$V_f$ is defined to be similar to~$\widehat{J}_f$, except that the measure~$\mu$ is replaced by the Hausdorff $\al$-measure~$\mathcal{H}_{\al}$.
  This is where the assumptions on $X$ are needed.
We start by recalling the result from~\cite{HK98}.
    \begin{thm}\textbf{\cite[Theorem 7.11]{HK98}}\label{7.11}
    Suppose that $(X,d_X,\mathcal{H}_{\alpha})$ and $(Y,d_Y,\mathcal{H}_{\alpha})$ are two $\al$-regular metric measure spaces for some~$\al > 1$, equipped with the Hausdorff $\al$-measure $\mathcal{H}_{\alpha}$. Suppose further that\\
    \indent\textup{(i)} $X$ and~$Y$ are rectifiably connected, \\
    \indent\textup{(ii)} $X$ and~$Y$ are locally compact, and\\
    \indent\textup{(iii)} $X$ admits a weak $(1,p)$-Poincar\'{e} inequality for some~$p$ with $1 \leq p < \al$.\\
Let $f$ be an $\eta$-quasisymmetric map from $X$ onto $Y$.
For $\mathcal{H}_{\alpha}$-a.e.~$x \in X$ and for~$r >0$ define the volume derivative
\[V_f(x) := \lim_{r \rightarrow 0^+} \frac{\mathcal{H}_{\alpha}(f(\overline{B}(x,r)))}{\mathcal{H}_{\alpha}(\overline{B}(x,r))}.\]
Then the pull-back measure $\sigma_f$, defined by
    \begin{equation}\label{hk1}
      \sigma_f(E) := \mathcal{H}_{\alpha}(f(E)), \text{ where } E \subset X,
    \end{equation}
    is $A_{\infty}$-related to the Hausdorff $\alpha$-measure $\mathcal{H}_{\alpha}$ in $X$. Moreover, $d\sigma_f = V_f \,d\mathcal{H}_{\alpha}$
    with $V_f(x) > 0$ for $\mathcal{H}_{\alpha}$-a.e. $x$ in $X$, and there is $\varepsilon >0$ such that
    \begin{equation}\label{hk3}
      \bigg(\intav_{B} V_f^{1+\varepsilon} \,d\mathcal{H}_{\alpha}\bigg)^{1/(1+\varepsilon)} \leq C\intav_{B} V_f \,d\mathcal{H}_{\alpha}
    \end{equation}
    for all balls $B$ in $X$. The statement is quantitative in that all the constants involved in the conclusion depend only on the quasisymmetry function of $f$, on the constants associated with the $\alpha$-regularity of $X$ and $Y$, and on the constant appearing in the Poincar\'{e} inequality.
    \end{thm}

Theorem~\ref{7.11} concludes that the volume derivative~$V_f$ is a reverse-H\"{o}lder weight.
Using Theorem~\ref{7.11}, we will develop in Theorem~\ref{thm:J_f in RHr} below an analogous result which shows that the Jacobian determinant~$\widehat{J}_f$ is a reverse-H\"{o}lder weight.
Note that we only require the analogue of one of the three conclusions in Theorem~\ref{thm:J_f in RHr}, namely inequality~\eqref{eq:hk3}, to prove Theorem~\ref{thm.R1.metric}.
\begin{thm}\textbf{(A generalisation of Theorem 7.11 in \cite{HK98})}\label{thm:J_f in RHr}
    Under the same conditions as in Theorem~\ref{thm.R1.metric},
    the measure~$\mu_f$ is $A_{\infty}$-related to the measure~$\mu$, where~$\mu_f := \mu(f(E))$ for all $\mu$-measurable sets~$E \subset X$.
     Moreover, for each $\mu$-measurable set~$S \subset X$,
   \begin{equation}\label{eq2:HK3}
    \mu_f(S) = \int_S \widehat{J}_f(x) \,d\mu(x),
   \end{equation}
    and there is $\varepsilon >0$ such that
    \begin{equation}\label{eq:hk3}
      \bigg(\intav_{B} \widehat{J}_f^{1+\varepsilon} \,d\mu\bigg)^{1/(1+\varepsilon)} \leq C\intav_{B} \widehat{J}_f \,d\mu
    \end{equation}
    for all balls $B$ in $X$. The statement is quantitative in that all the constants involved in the conclusion depend only on the quasisymmetry function of $f$, on the constants associated with the $\alpha$-regularity of $X$, and on the constant appearing in the Poincar\'{e} inequality.
\end{thm}
\begin{proof}
  We start by proving the most important conclusion of Theorem~\ref{thm:J_f in RHr}, which shows that~$\widehat{J}_f$ is a reverse-H\"{o}lder weight. The main ingredients are inequality~\eqref{hk3}, and the comparability of~$\mu$ and~$\mathcal{H}_{\al}$, of~$\mu_f$ and~$\sigma_f$, and of~$\widehat{J}_f$ and~$V_f$ shown in Lemma~\ref{lem:mu_f compar sigma_f} below.

\begin{lem}\label{lem:mu_f compar sigma_f}
  Under the same conditions as in Theorem~\ref{thm.R1.metric},
  the measure~$\mu_f$ is comparable to the measure~$\sigma_f$ defined in~\eqref{hk1}.
  Moreover, the Jacobian determinant~$\widehat{J}_f$ and the volume derivative~$V_f$ are comparable.
\end{lem}
\begin{proof}
  Recall that for Borel sets~$E \subset X$, the measures~$\mu_f$ and~$\sigma_f$ are defined by~$\mu_f(E) := \mu(f(E))$ and~$\sigma_f(E) := \mathcal{H}_{\al}(f(E))$.
   Note that under the same conditions as in Theorem~\ref{thm.R1.metric}, $\widehat{J}_f(x)$ exists and is finite for $\mu$-a.e.~$x \in X$ (Lemma~\ref{lem:J_f_hat exist}). The same proof can be used to show that~$V_f(x)$  exists and is finite for $\mathcal{H}_{\al}$-a.e.~$x \in X$, by replacing~$\mu$ by~$\mathcal{H}_{\al}$ and~$\mu_f$ by~$\sigma_f$.

  Recall that~$\mu$ is comparable to~$\mathcal{H}_{\al}$, say with constant~$C_5$. Moreover, as~$f$ is a homeomorphism, for each Borel set~$E \subset X$, $f(E)$ is also a Borel set in~$X$.
   Therefore~$\mu_f$ is comparable to~$\sigma_f$ also with constant~$C_5$.

   Using the comparability of~$\mu$ and~$\mathcal{H}_{\al}$ and of~$\mu_f$ and~$\sigma_f$, we can easily see that~$\widehat{J}_f$ is comparable to~$V_f$ with constant~$C_5^2$:
   \[C_5^{-2}\lim_{r \rightarrow 0^+}\frac{\mu_f(\overline{B}(x,r))}{\mu(\overline{B}(x,r))}
   \leq \lim_{r \rightarrow 0^+} \frac{\sigma_f(\overline{B}(x,r)}{\mathcal{H}_{\alpha}(\overline{B}(x,r))}
   \leq C_5^2 \lim_{r \rightarrow 0^+}\frac{\mu_f(\overline{B}(x,r))}{\mu(\overline{B}(x,r))}. \qedhere\]
\end{proof}

Returning to the proof of Theorem~\ref{thm:J_f in RHr}, now we are ready to show inequality~\eqref{eq:hk3}.
Fix a ball~$B \subset X$. Take~$\e > 0$ from Theorem~\ref{7.11} and the constant~$C$ from~\eqref{hk3}.
Using  the fact that~$\mu$ and~$\mathcal{H}_{\al}$ are comparable with constant~$C_5$, Proposition~\ref{prop:properties_meable_func}(iii), Lemma~\ref{lem:mu_f compar sigma_f}, and inequality~\eqref{hk3} from Theorem~\ref{7.11}, we obtain
\begin{eqnarray*}
  \bigg(\frac{1}{\mu(B)}\int_{B} \widehat{J}_f^{1+\varepsilon} \,d\mu\bigg)^{1/(1+\varepsilon)}
   &\leq& C_5^{2/(1+\e)}  \bigg( \frac{1}{\mathcal{H}_{\al}(B)} \int_B \widehat{J}_f^{1+\varepsilon} \,d\mathcal{H}_{\al} \bigg)^{1/(1+\varepsilon)} \\
   &\leq& C_5^{2/(1+\e)} C_5^2  \bigg( \frac{1}{\mathcal{H}_{\al}(B)} \int_B V_f^{1+\e} \,d\mathcal{H}_{\al} \bigg)^{1/(1+\varepsilon)} \\
   &\leq& C C_5^{2/(1+\e)} C_5^4   \frac{1}{\mathcal{H}_{\al}(B)} \int_B \widehat{J}_f \,d\mathcal{H}_{\al}\\
   &\leq& C C_5^{2/(1+\e)} C_5^6   \frac{1}{\mu(B)}\int_B  \widehat{J}_f \,d\mu.
\end{eqnarray*}
Thus we have established inequality~\eqref{eq:hk3}, with a constant depending only on the constants in Theorem~\ref{7.11} and  the comparability of~$\mu$ and~$\mathcal{H}_{\al}$.

The first result of Theorem~\ref{thm:J_f in RHr} follows from the facts that $\mu_f$ is comparable to~$\sigma_f$ and~$\mu_f$ is $A_{\infty}$-related to~$\sigma_f$, and the properties of $A_{\infty}$-relatedness shown in Proposition~\ref{prop:properties_A_infty}.
The second result of Theorem~\ref{thm:J_f in RHr}, which is equation~\eqref{eq2:HK3}, is the same as the last conclusion of the Radon--Nikodym Theorem (Theorem~\ref{lem:Nages analog}), in the special case of a metric measure space, when the measure~$\nu$ is replaced by~$\mu_f$.
We omit the details.
\end{proof}
\subsection{$\log \widehat{J}_f \in \bmo$}\label{subsec:log_Jfhat_BMO}
We have shown that under the hypotheses of Theorem~\ref{thm.R1.metric}, the generalised Jacobian~$\widehat{J}_f$ is a reverse-H\"{o}lder weight on~$(X,d,\mu)$.
Finally, by applying Theorem~\ref{thm.logBMO} with $w = \widehat{J}_f$, we conclude that~$\log \widehat{J}_f \in \bmo(X,d,\mu)$. This completes the proof of Theorem~\ref{thm.R1.metric}.

\begin{rem}\label{rem:hypo_use}
Hypotheses~(i)--(iv) come from~\cite[Theorem~7.11]{HK98}, as we use this theorem in Step~3 of our proof to show that $\widehat{J}_f \in RH_q(X)$. Similarly, these hypotheses are also required in Corollary~\ref{thm.R1.quasimetric} (see hypotheses~(a)--(d))
and in Example~\ref{thm:construct} below (see hypotheses~(i)--(iv)).

We also note that under the assumptions of Theorem~\ref{thm.R1.metric}, the space~$X$ is rectifiably connected. Please refer to Remark~\ref{rem:recti_connect} for more details.
\end{rem}

\begin{rem}\label{rem:alter_hypo}
We recall Theorem~1.1 in~\cite{KZ08}, which shows that given a complete metric measure space~$(X,d,\mu)$ with~$\mu$ being Borel and doubling and~$X$ admitting a $(1,\al)$-Poincar\'{e} inequality, where~$\al > 1$, there exists~$\e>0$ such that~$(X,d,\mu)$ admits a $(1,p)$-Poincar\'{e} inequality for every~$p > \al-\e$, quantitatively.
In our Theorem~\ref{thm.R1.metric},
it may be possible to replace hypothesis~(iv) by the weaker assumption that~$X$ admits a weak $(1,\al)$-Poincar\'{e} inequality, where~$\al > 1$ is from hypothesis~(iii), using Theorem~1.1 in~\cite{KZ08}. In that case, we would have to impose the additional assumption that~$X$ is complete.

We also recall Theorem~1.3 in~\cite{KZ08}, which shows that a complete $\al$-regular metric measure space, with $\al >1$, is a Loewner space if and only if it supports a $(1, \al - \e)$-Poincar\'{e} inequality for some~$\e > 0$, quantitatively.
In our Theorem~\ref{thm.R1.metric},
it may be possible to remove hypothesis~(iv), using Theorem~1.3 in~\cite{KZ08}. In that case, we would have to impose the additional assumptions that~$X$ is complete and is a Loewner space.
\end{rem}


\section{Proof of Corollary~\ref{thm.R1.quasimetric}}\label{sec:R1_quasimec}
In this section, we establish our third main result, which is an analogue of Reimann's Theorem~1 on spaces of homogeneous type~$(X,\rho,\mu)$, stated in Corollary~\ref{thm.R1.quasimetric}. 
The seven steps of the proof of Corollary~\ref{thm.R1.quasimetric} are outlined in the Introduction.
To complete the proof of Corollary~\ref{thm.R1.quasimetric}, our task is to prove the lemmas mentioned there, as well as to complete Step~6.

Before that, we prove Proposition~\ref{prop:compare_ball_quasiball} which is useful later. Proposition~\ref{prop:compare_ball_quasiball} says that the quasiballs~$\widetilde{B}$ on~$(X,\rho,\mu)$ are comparable to the (metric) balls~$\widehat{B}$ on~$(X,d_{\e},\mu)$, where~$d_{\e}$ is a metric which is comparable to the snowflaking~$\rho_{\e}$ of the quasimetric~$\rho$.

\begin{prop}\label{prop:compare_ball_quasiball}
Suppose~$(X,\rho,\mu)$ is a space of homogeneous type.
Given~$\e \in (0,1]$, let $\rho_{\e}(x,y) := \rho(x,y)^{\e}$ for all~$x,y \in X$.
Let~$d_{\e}$ be a metric which is comparable to~$\rho_{\e}$ with constant~$C_{\e} \geq 1$.
Then for all~$x \in X$ and~$r > 0$,
\begin{equation}\label{eq:Hei6}
  \widehat{B}(x,\Cc^{-1}r^{\e}) \subset \widetilde{B}(x,r) \subset \widehat{B}(x,\Cc r^{\e}). 
\end{equation}

For each~$x \in X$ and~$r > 0$, let~$\overline{B}_{\rho}(x,r)$ and~$\overline{B}_{d_{\e}}(x,r)$ be the closed balls on~$(X,\rho,\mu)$ and~$(X,d_{\e},\mu)$, respectively. Then it also holds that
\begin{equation}\label{eq:Hei6a}
  \overline{B}_{d_{\e}}(x,\Cc^{-1}r^{\e}) \subset \overline{B}_{\rho}(x,r) \subset \overline{B}_{d_{\e}}(x,\Cc r^{\e}). 
\end{equation}
\end{prop}

\begin{proof}
Fix~$x \in X$ and~$r >0$. Fix~$y \in \widehat{B}(x,\Cc^{-1}r^{\e})$.
Then~$d_{\e}(x,y) < \Cc^{-1}r^{\e}$. Since~$d_{\e} \sim_{C_{\e}} \rho_{\e}$, we have $\rho_{\e}(x,y)< r^{\e}$, and so~$\rho(x,y) < r$. Therefore, $y \in \widetilde{B}(x,r)$. Hence the first inclusion of~\eqref{eq:Hei6} holds.

The second inclusion of~\eqref{eq:Hei6} can be proved analogously, completing the proof of Proposition~\ref{prop:compare_ball_quasiball}.

Property~\eqref{eq:Hei6a} can be shown using the same proof structure.
\end{proof}

Now we will state and prove Lemmas~\ref{lem:loccompequi},
\ref{lem:loccompactequi}, \ref{lem:homeoequi}, \ref{lem:J_f_tilde exist} and \ref{lem:Jacoequi}.
\subsection{Passing from~$\rho$ to~$d_{\e}$ preserves local compactness}\label{subsec:imply loccom}
\begin{lem}\label{lem:loccompequi}
  Suppose~$(X,\rho,\mu)$ is a space of homogeneous type.
  Given~$\e \in (0,1]$, let~$\rho_{\e}(x,y) := \rho(x,y)^{\e}$ for all~$x,y \in X$.
  Let~$d_{\e}$ be a metric which is comparable to~$\rho_{\e}$ with constant~$C_{\e} \geq 1$.
  Then~$(X,\rho,\mu)$ is locally compact if and only if~$(X,d_{\e},\mu)$ is also locally compact.
\end{lem}
\begin{proof}
We will show that if~$(X,\rho,\mu)$ is locally compact then~$(X,d_{\e},\mu)$ is also locally compact. The proof of the reverse direction is similar.

Suppose the space of homogeneous type~$(X,\rho,\mu)$ is locally compact. That is, for all~$x \in X$, there exist an open set~$O$ w.r.t.~$\rho$ and a compact set~$K$ w.r.t.~$\rho$ such that~$x \in O \subset K$.

A set~$O$ is \emph{open w.r.t.~$\rho$} means for all~$x \in O$, there exists a quasiball~$\widetilde{B}(x,r_x)$ centred at~$x$ such that~$x \in \widetilde{B}(x,r_x) \subset O$.
A set~$K$ is \emph{compact w.r.t.~$\rho$} means every open cover w.r.t.~$\rho$ of~$K$ has a finite subcover.
The definitions of \emph{open sets w.r.t.~$d_{\e}$} and \emph{compact sets w.r.t.~$d_{\e}$} are analogous, except that the quasiball~$\widetilde{B}$ is replaced by the ball~$\widehat{B}$.

We want to show that~$(X,d_{\e},\mu)$ is locally compact. That is, for all~$x \in X$, there exist an open set~$O$ w.r.t.~$d_{\e}$ and a compact set~$K$ w.r.t.~$d_{\e}$ such that~$x \in O \subset K$.

%

Fix a point~$x \in X$. Since~$(X,\rho,\mu)$ is locally compact, there exist an open set~$O$ w.r.t.~$\rho$ and a compact set~$K$ w.r.t.~$\rho$ such that~$x \in O \subset K$.
We claim that the set~$O$ is also open w.r.t~$d_{\e}$. This will be proved in Claim~\ref{claim:rhoOpen_dOpen} below. Moreover, the set~$K$ is also compact w.r.t.~$d_{\e}$. This will be shown in Claim~\ref{claim:rhoCompact_dCompact} below. Since this is true for all~$x \in X$, we conclude that~$(X,d_{\e},\mu)$ is locally compact. The proofs of Claims~\ref{claim:rhoOpen_dOpen} and~\ref{claim:rhoCompact_dCompact} below complete the proof of Lemma~\ref{lem:loccompequi}.

\begin{claim}\label{claim:rhoOpen_dOpen}
     A set~$O$ that is open w.r.t.~$\rho$ is also open w.r.t.~$d_{\e}$.
\end{claim}
\noindent\emph{Proof of Claim~\ref{claim:rhoOpen_dOpen}:}
Since~$O$ is open w.r.t.~$\rho$, for each~$y \in O$, there exists a quasiball~$\widetilde{B}(y,r_y)$ centred at~$y$ such that~$y \in \widetilde{B}(y,r_y) \subset O$.
Using property~\eqref{eq:Hei6} we have
$y \in \widehat{B}(y, \Cc^{-1}r_y^{\e}) \subset \widetilde{B}(y,r_y) \subset O$.
Thus, for each~$y \in O$, there exists a (metric) ball~$\widehat{B}_y$ centred at~$y$ such that~$y \in \widehat{B}_y \subset O$. Claim~\ref{claim:rhoOpen_dOpen} is established. \hfill\(\Box\)

\begin{claim}\label{claim:rhoCompact_dCompact}
A set~$K$ that is compact w.r.t.~$\rho$ is also compact w.r.t.~$d_{\e}$.
\end{claim}
\noindent\emph{Proof of Claim~\ref{claim:rhoCompact_dCompact}:}
Let~$\{O_{\al}\}$ be an open cover w.r.t.~$\rho$ of~$K$. So each~$O_{\al}$ is open w.r.t.~$\rho$ and~$K \subset \cup_{\al} O_{\al}$.
Hence by Claim~\ref{claim:rhoOpen_dOpen}, each~$O_{\al}$ is also open w.r.t.~$d_{\e}$, and so~$\{O_{\al}\}$ is also an open cover w.r.t.~$d_{\e}$ of~$K$. In addition, since~$K$ is compact w.r.t.~$\rho$, there exists a finite subcollection~$\{O_1, O_2, \ldots, O_n\}$ such that~$K \subset \cup_{i =1}^n O_i$, and therefore~$K$ is  also compact w.r.t.~$d_{\e}$. \hfill\(\Box\)

This concludes the proof of Lemma~\ref{lem:loccompequi}.
\end{proof}

\subsection{Passing from~$\rho$ to~$d_{\e}$ preserves Ahlfors regularity}\label{subsec:imply alregular}
\begin{lem}\label{lem:loccompactequi}
  Suppose~$(X,\rho,\mu)$ is a space of homogeneous type.
  Given~$\e \in (0,1]$, let~$\rho_{\e}(x,y) := \rho(x,y)^{\e}$ for all~$x,y \in X$.
  Let~$d_{\e}$ be a metric which is comparable to~$\rho_{\e}$ with constant~$C_{\e} \geq 1$.
  Then~$(X,\rho,\mu)$ is an~$\al$-regular space with constant~$\kappa$ if and only if~$(X,d_{\e},\mu)$ is an $\al/\e$-regular space with constant~$\kappa_0 = A_1^{1+\log_2\Cc} \kappa$.
\end{lem}
\begin{proof}
We will show that if~$(X,\rho,\mu)$ is an~$\al$-regular space, then~$(X,d_{\e},\mu)$ is an $\al/\e$-regular space. The proof of the reverse direction is similar.
We recall the definition of $\al$-regular space in Definition~\ref{def:alpha},
and the result shown in~Proposition~\ref{prop:compare_ball_quasiball}.
These together with the doubling property of~$\mu$ shown in~\eqref{eq:dbl_measure_2} give us
\begin{eqnarray*}
 A_1^{-(1+\log_2\Cc)} \mu(\widehat{B}(x,r^{\e}))
 &\leq& \mu(\widehat{B}(x,\Cc^{-1}r^{\e}))
 \leq \mu(\widetilde{B}(x,r))
  \leq \kappa r^{\alpha},
\end{eqnarray*}
and
\begin{equation*}
  \kappa^{-1}r^{\alpha} \leq \mu(\widetilde{B}(x,r))
  \leq \mu(\widehat{B}(x,\Cc r^{\e}))
  \leq A_1^{1+\log_2\Cc}\mu(\widehat{B}(x, r^{\e})).
\end{equation*}
It follows that
\[A_1^{-(1+\log_2\Cc)} \kappa^{-1}r^{\alpha} \leq \mu(\widehat{B}(x, r^{\e})) \leq A_1^{1+\log_2\Cc} \kappa r^{\alpha}.\]
Hence, the space~$(X,d_{\e},\mu)$ is $\al/\e$-regular with the constant~$\kappa_0 = A_1^{1+\log_2\Cc} \kappa.$
\end{proof}

\subsection{Passing from~$\rho$ to~$d_{\e}$ preserves quasisymmetry of functions}\label{subsec:imply_quasisym}
\begin{lem}\label{lem:homeoequi}
  Suppose~$(X,\rho,\mu)$ is a space of homogeneous type.
  Given~$\e \in (0,1]$, let~$\rho_{\e}(x,y) := \rho(x,y)^{\e}$ for all~$x,y \in X$.
  Let~$d_{\e}$ be a metric which is comparable to~$\rho_{\e}$ with constant~$C_{\e} \geq 1$.
  If a homeomorphism~$f$ is $\eta$-quasisymmetric from~$(X,\rho,\mu)$ onto itself, then~$f$ is $\zeta$-quasisymmetric from~$(X,d_{\e},\mu)$ onto itself, where~$\zeta: [0,\infty)\rightarrow [0,\infty)$ is the homeomorphism defined by
  $\zeta(\theta) := \Cc^2\eta([\Cc^2\theta]^{1/\e})^{\e}.$
\end{lem}
\begin{proof}
We recall the definition of $\eta$-quasisymmetric maps in Definition~\ref{def:quasisym}.
It suffices to show that there exists an increasing homeomorphism~$\zeta: [0,\infty)\rightarrow [0,\infty)$ so that for all~$\theta \geq 0$ and all distinct $x, a, b \in X$ we have that for all~$x \in X$ and~$r > 0$,
\begin{equation}\label{eq1:homeoequi}
  \frac{d_{\e}(x,a)}{d_{\e}(x,b)}  \leq \theta \qquad \Rightarrow \qquad
  \frac{\rho(x,a)}{\rho(x,b)} \leq (\Cc^2\theta)^{1/\e}
\end{equation}
and
\begin{equation}\label{eq2:homeoequi}
  \frac{\rho(f(x),f(a))}{\rho(f(x),f(b))} \leq \eta([\Cc^2\theta]^{1/\e})
   \quad \Rightarrow \quad \frac{d_{\e}(f(x),f(a))}{d_{\e}(f(x),f(b))} \leq \zeta(\theta).
\end{equation}
We start with~\eqref{eq1:homeoequi}.
Fix~$\theta \geq 0$. Suppose~$x,a,b \in X$ are distinct points so that $d_{\e}(x,a)/d_{\e}(x,b) \leq \theta$. Since~$d_{\e} \sim_{C_{\e}} \rho_{\e}$, we have
\[\frac{\rho(x,a)}{\rho(x,b)}
=  \bigg(\frac{\rho_{\e}(x,a)}{\rho_{\e}(x,b)}\bigg)^{1/\e}
\leq \bigg(C_{\e}^2\frac{d_{\e}(x,a)}{d_{\e}(x,b)}\bigg)^{1/\e}
\leq (\Cc^2\theta)^{1/\e}, \text{ as required}.\]

Following the same structure, we find that property~\eqref{eq2:homeoequi} holds with
 $\zeta(\theta) := \Cc^2\eta([\Cc^2\theta]^{1/\e})^{\e}.$

We recall that the composition of homeomorphisms is also a homeomorphism, and the composition of increasing functions is also an increasing function. The function~$\zeta$ is a composition of increasing homeomorphisms. Therefore, $\zeta$ is an increasing homeomorphism from~$[0,\infty)$ onto itself.

Combining~\eqref{eq1:homeoequi}, \eqref{eq2:homeoequi} and the fact that~$f$ is $\eta$-quasisymmetric from~$(X,\rho,\mu)$ onto itself, we see that~$f$ is also $\zeta$-quasisymmetric from~$(X,d_{\e},\mu)$ onto itself.
\end{proof}

\subsection{Existence of~$\widetilde{J}_f$}\label{subsec:Jftilde exist}
\begin{lem}\label{lem:J_f_tilde exist}
  Under the same conditions as in Corollary~\ref{thm.R1.quasimetric}, the generalised Jacobian~$\widetilde{J}_f$ exists and is finite for $\mu$-a.e.~$x \in X$.
\end{lem}

\begin{proof}
  For each~$x \in X$ and~$r >0$, let~$\overline{B}_{\rho}(x,r) := \{y \in X: \rho(x,y) \leq r\}$  denote the closed quasiball centred at~$x$ of radius~$r$ in~$(X,\rho,\mu)$.
  Notice that for all~$x \in X$ the generalised Jacobian~$\widetilde{J}_f$ coincides with the Radon--Nikodym derivative of~$\mu_f$ w.r.t.~$\mu$:
    $$\widetilde{D}(\mu_f,\mu,x):= \lim_{r \rightarrow 0^+ } \frac{\mu_f(\overline{B}_{\rho}(x,r))}{\mu(\overline{B}_{\rho}(x,r))}
    = \widetilde{J}_f(x).$$
  Consequently,~$\widetilde{J}_f(x)=\widetilde{D}(\mu_f,\mu,x)$ for all~$x \in X$.
  Besides this, the absolute continuity of~$\mu_f$ w.r.t.~$\mu$ together with the fact that the measure~$\mu$ and~$\mu_f$ are both Borel regular allow us to use our Radon--Nikodym Theorem (Theorem~\ref{lem:Nages analog}) to conclude that
  $\widetilde{D}(\mu_f,\mu,x)$
  exists  and is finite for $\mu$-a.e.~$x \in X$. Hence~$\widetilde{J}_f(x)$ exists and is finite for $\mu$-a.e.~$x \in X$.
\end{proof}

\subsection{$\widehat{J}_f$ and $\widetilde{J}_f$ are comparable}\label{subsec: Jacoequi}
\begin{lem}\label{lem:Jacoequi}
Under the same conditions as in Corollary~\ref{thm.R1.quasimetric},
the generalised Jacobians $\widehat{J}_f$ and~$\widetilde{J}_f$ are comparable with a constant depending on~$A_1, C_{\mu_f}, \Cc$ and~$\e$.
\end{lem}

\begin{proof}
Recall that~$C_{\mu_f}$ is the doubling constant of the measure~$\mu_f$.
We have shown in Lemmas~\ref{lem:J_f_hat exist} and~\ref{lem:J_f_tilde exist} that $\widehat{J}_f(x)$ and $\widetilde{J}_f(x)$ exist for $\mu$-a.e.~$x \in X$.
Take~$x \in X$ such that $\widehat{J}_f(x)$ and $\widetilde{J}_f(x)$ both exist.
Using Proposition~\ref{prop:compare_ball_quasiball} and the doubling properties of~$\mu$ and~$\mu_f$ we have
\begin{eqnarray*}
  \widehat{J}_f(x) &=& \lim_{r \rightarrow 0} \frac{\mu_f(\overline{B}_{d_{\e}}(x, r)))}{\mu(\overline{B}_{d_{\e}}(x,r)} \\
   &\leq&  \lim_{r \rightarrow 0}\frac{\mu_f(\overline{B}_{\rho}(x,\Cc^{1/\e}r^{1/\e}))}
   {\mu(\overline{B}_{\rho}(x,\Cc^{-1/\e}r^{1/\e}))} \\
    &\leq&  \lim_{r \rightarrow 0}\frac{\mu_f(\widetilde{B}(x,2\Cc^{1/\e}r^{1/\e}))}
   {\mu(\widetilde{B}(x,\Cc^{-1/\e}r^{1/\e}))} \\
    &\leq& \frac{C_{\mu_f}^{2+\log_2 \Cc^{1/\e}}}{A_1^{-(\log_2\Cc^{-1/\e})}}
   \lim_{r \rightarrow 0} \frac{\mu_f(\widetilde{B}(x,r^{1/\e}))}
   {\mu(\widetilde{B}(x,2r^{1/\e}))} \\
   &=& \frac{C_{\mu_f}^{2+\log_2 \Cc^{1/\e}}}{A_1^{-(\log_2\Cc^{-1/\e})}} \lim_{t \rightarrow 0}\frac{\mu_f(\widetilde{B}(x,t))}{\mu(\widetilde{B}(x,2t))} \qquad \text{where } t:=r^{1/\e}\\
   &=& \frac{C_{\mu_f}^{2+\log_2 \Cc^{1/\e}}}{A_1^{-(\log_2\Cc^{-1/\e})}} \lim_{t \rightarrow 0}\frac{\mu_f(\overline{B}_{\rho}(x,t))}{\mu(\overline{B}_{\rho}(x,t))}\\
   &=& \frac{C_{\mu_f}^{2+\log_2 \Cc^{1/\e}}}{A_1^{-(\log_2\Cc^{-1/\e})}} \widetilde{J}_f(x) = C(A_1, C_{\mu_f}, C_{\e},\e) \widetilde{J}_f(x).
\end{eqnarray*}

Similarly we have
\begin{equation*}
  \widetilde{J}_f(x)
   = \lim_{r \rightarrow 0}\frac{\mu_f(\overline{B}_{\rho}(x,r))}{\mu(\overline{B}_{\rho}(x,r)}
   \leq \frac{C_{\mu_f}^{2+\log_2 \Cc}}{A_1^{-(\log_2\Cc)}} \widehat{J}_f(x) = C(A_1, C_{\mu_f}, C_{\e},\e) \widehat{J}_f(x). \qedhere
\end{equation*}\qedhere
\end{proof}

Next, we move to Step~6 of the proof of Corollary~\ref{thm.R1.quasimetric}.
\subsection{$\log  \widetilde{J}_f \in \bmo(X,d_{\e},\mu)$}\label{subsec:logJacoequi}
In this section, we will show that under the same conditions as in Corollary~\ref{thm.R1.quasimetric}, $\log \widetilde{J}_f \in \bmo(\widehat{X})$.
We will use the result of the following Proposition.
\begin{prop}\label{prop:bmo equi}
Suppose~$(X,\rho,\mu)$ is a space of homogeneous type.
Let~$g: X \rightarrow \R$ and~$h: X \rightarrow \R$ be positive locally integrable functions.
If~$g \in \bmo(X,\rho,\mu)$ and there exists a constant~$C$ such that
  $|g(x) -  h(x)| \leq C$
  for $\mu$-a.e.~$x \in X$,
  then $h \in \bmo(X,\rho,\mu)$.
\end{prop}
\begin{proof}
  Fix a quasiball~$\widetilde{B} \subset X$. For each~$x \in \widetilde{B}$, consider
  \begin{equation}\label{eq1:bmo equi}
    |h(x) - h_{\widetilde{B}}| \leq |h(x) - g(x)| + |g(x) - g_{\widetilde{B}}| + |g_{\widetilde{B}} - h_{\widetilde{B}}|.
  \end{equation}
  The first and third terms on the right-hand side of~\eqref{eq1:bmo equi} are bounded above by~$C$:
  \[|h(x) - g(x)| \leq C, \quad \text{and}\]
  \begin{equation*}
    |g_{\widetilde{B}} - h_{\widetilde{B}}|
     = \bigg|\intav_{\widetilde{B}} g(x) - h(x) \,d\mu\bigg|
     \leq \intav_{\widetilde{B}} |g(x) - h(x)| \,d\mu  \leq C.
  \end{equation*}
  Hence~$\eqref{eq1:bmo equi}$ yields
  $|h(x) - h_{\widetilde{B}}| \leq 2C + |g(x) - g_{\widetilde{B}}|.$
  Taking the average over~$\widetilde{B}$ on both sides we get
  \[\intav_{\widetilde{B}} |h(x) - h_{\widetilde{B}}|\,d\mu \leq 2C + \intav_{\widetilde{B}} |g(x) - g_{\widetilde{B}}| \,dx
  \leq 2C + \|g\|_{\bmo(X,\rho,\mu)}.\]
  Since this is true for all quasiballs~$\widetilde{B} \subset X$,
  the function $h$ is in $\bmo(X,\rho,\mu)$ with $\|h\|_{\bmo} \leq 2C + \|g\|_{\bmo}$.
\end{proof}

Under the same conditions as in Corollary~\ref{thm.R1.quasimetric},
by passing from the quasimetric~$\rho$ to the metric~$d_{\e}$
we obtain a metric measure space~$\widehat{X} = (X,d_{\e},\mu)$ that satisfies the conditions of Theorem~\ref{thm.R1.metric}. This implies that the generalised Jacobian~$\widehat{J}_f$ exists and is finite, and more importantly, $\log \widehat{J}_f \in \bmo(\widehat{X})$.

From Lemma~\ref{lem:Jacoequi} we know that there exists a constant~$C$ such that $C^{-1}\widehat{J}_f \leq \widetilde{J}_f \leq C\widehat{J}_f$,
where $C = C(A_1, C_{\mu_f}, C_{\e},\e)$.
Then we have
\[\log  \widehat{J}_f  - C \leq \log  \widetilde{J}_f  \leq \log  \widehat{J}_f  + C \quad \mu\text{-a.e.},\]
which is equivalent to $|\log \widehat{J}_f - \log \widetilde{J}_f| \leq C$ $\mu$-a.e.
Applying Proposition~\ref{prop:bmo equi} for~$g(x) = \log  \widehat{J}_f(x)$ and~$h(x) = \log  \widetilde{J}_f(x)$, we conclude that~$\log  \widetilde{J}_f \in \bmo(\widehat{X})$ with~$\|\log  \widetilde{J}_f \|_{\bmo(\widehat{X})} \leq 2C + \|\widehat{J}_f\|_{\bmo(\widehat{X})}$.

\subsection{$\bmo(X,\rho,\mu)$ and $\bmo(X,d_{\e},\mu)$ coincide}\label{subsec:BMO_coincide}
In this section, we show that $\bmo(X,\rho,\mu) = \bmo(X,d_{\e},\mu)$, where~$(X,\rho,\mu)$ is a space of homogeneous type, and~$d_{\e}$ is a  metric which is comparable to the snowflaking~$\rho_{\e}$ of the quasimetric~$\rho$.
See Section~\ref{subsec:BMO_X} for the definitions of these~$\bmo$ spaces.

\begin{prop}\label{prop:BMO_coincide}
  Let~$(X,\rho,\mu)$ be a space of homogeneous type.
  Given~$\e \in (0,1]$, let~$\rho_{\e}(x,y) := \rho(x,y)^{\e}$ for all~$x,y \in X$.
  Let~$d_{\e}$ be a metric which is comparable to~$\rho_{\e}$ with constant~$C_{\e} \geq 1$.
  Then $\bmo(X,\rho,\mu) = \bmo(X,d_{\e},\mu)$, with comparable norms.
\end{prop}
\begin{proof}
  Let $\widetilde{X} := (X,\rho,\mu)$ and $\widehat{X} := (X,d_{\e},\mu)$.
  It is sufficient to show that there exist constants~$C>0$ and~$C'>0$ depending on~$\e$ such that for every~$\varphi \in L^1_{\text{loc}}(\widetilde{X})$, there holds
  \[
  C\|\varphi\|_{\bmo(\widehat{X})}
  \leq \|\varphi\|_{\bmo(\widetilde{X})} \leq
  C'\|\varphi\|_{\bmo(\widehat{X})}.
  \]
  Fix an~$x \in X$ and~$r>0$. Let~$\widetilde{B} := \widetilde{B}(x,r)$, and~$\widehat{B} := \widehat{B}(x,C_{\e}r^{\e})$.
  Then
  \begin{equation}\label{eq1:BMO_coincide}
    \intav_{\widetilde{B}} |\varphi(x) - \varphi_{\widetilde{B}}| \,d\mu
    \leq \intav_{\widetilde{B}} |\varphi(x) - \varphi_{\widehat{B}}| \,d\mu
    + \intav_{\widetilde{B}} |\varphi_{\widehat{B}} - \varphi_{\widetilde{B}}| \,d\mu.
  \end{equation}
  We consider each integral on the right-hand side of~\eqref{eq1:BMO_coincide}.
  Using the nestedness property~\eqref{eq:Hei6} and the doubling property~\eqref{eq:dbl_measure_2} of~$\mu$ we have
  \begin{equation}\label{eq2:BMO_coincide}
   \frac{1}{\widetilde{B}} \int_{\widetilde{B}} |\varphi(x) - \varphi_{\widehat{B}}| \,d\mu
    \leq \frac{A_1^{1 + \log_2 C_{\e}^{2}}}{\mu(\widehat{B})} \int_{\widehat{B}}  |\varphi(x) - \varphi_{\widehat{B}}| \,d\mu
    \leq A_1^{1 + \log_2 C_{\e}^{2}} \|\varphi\|_{\bmo(\widehat{X})}.
  \end{equation}
  Moreover, using~\eqref{eq2:BMO_coincide} we obtain
  \[
  |\varphi_{\widehat{B}} - \varphi_{\widetilde{B}}|
  \leq \intav_{\widetilde{B}} |\varphi(x) - \varphi_{\widehat{B}}| \,d\mu
  \leq A_1^{1 + \log_2 C_{\e}^{2}} \|\varphi\|_{\bmo(\widehat{X})},
  \]
  and so
  \begin{equation}\label{eq5:BMO_coincide}
  \intav_{\widetilde{B}} |\varphi_{\widehat{B}} - \varphi_{\widetilde{B}}| \,d\mu
  \leq A_1^{1 + \log_2 C_{\e}^{2}} \|\varphi\|_{\bmo(\widehat{X})}.
  \end{equation}
  Taking the supremum over all quasiballs~$\widetilde{B} \subset \widetilde{X}$ of~\eqref{eq1:BMO_coincide}, and using~\eqref{eq2:BMO_coincide} and~\eqref{eq5:BMO_coincide} we obtain
    \begin{equation}\label{eq3:BMO_coincide}
     \|\varphi\|_{\bmo(\widetilde{X})}
     = \sup_{\widetilde{B}}  \intav_{\widetilde{B}} |\varphi(x) - \varphi_{\widetilde{B}}| \,d\mu
    \leq 2 A_1^{1 + \log_2 C_{\e}^{2}} \|\varphi\|_{\bmo(\widehat{X})}.
  \end{equation}

  Following the same argument, we have
  \begin{equation}\label{eq4:BMO_coincide}
     \|\varphi\|_{\bmo(\widetilde{X})}
    \geq 2 A_1^{-(1 + \log_2 C_{\e}^{2})} \|\varphi\|_{\bmo(\widehat{X})}.
  \end{equation}
  Combining~\eqref{eq3:BMO_coincide} and~$\eqref{eq4:BMO_coincide}$, Proposition~\ref{prop:BMO_coincide} is established.
\end{proof}

 Using Proposition~\ref{prop:BMO_coincide}
 we conclude that $\widetilde{J}_f \in \bmo({\widetilde{X}})$,
 completing the proof of Corollary~\ref{thm.R1.quasimetric}.
\section{Construction of Spaces~$(X,\rho,\mu)$ to which our Result Applies}\label{sec:construction_X}
In this section, we construct a large class of spaces of homogeneous type to which Corollary~\ref{thm.R1.quasimetric} applies.
The idea is that we start with any metric measure space~$(X,D,\mu)$ satisfying the conditions of Theorem~$\ref{thm.R1.metric}$.
From there, we can always build a class of spaces of homogeneous type~$(X,\rho,\mu)$ such that the hypotheses of Corollary~\ref{thm.R1.quasimetric} hold.
The detail of the construction is shown in Example~\ref{thm:construct}. The four main steps of its proof are outlined in the Introduction.
We are left to prove Lemma~\ref{lem:construct1} from Step~2, and to complete Step~4.
Recall that given a metric measure space~$(X,D,\mu)$ as in Example~\ref{thm:construct}, we fix~$\beta \geq 1$, and define~$\rho (x,y) := D(x,y)^{\beta}$ for all~$x, y \in X$.
In Lemma~\ref{lem:construct1}, we show that~$\rho$ is in fact a quasimetric.

\begin{lem}\label{lem:construct1}
Let~$D$ be a metric on a set~$X$. Suppose~$\beta \geq 1$. Define~$\rho (x,y) := D(x,y)^{\beta}$ for all~$x, y \in X$. Then~$\rho$ is a quasimetric on~$X$, with quasitriangle constant~$A_0 = 2^{\beta -1}$.
\end{lem}

\begin{proof}
Take~$x, y, z \in X$.
  First,
  $\rho(x,y) = 0 \Leftrightarrow D(x,y)^{\beta} = 0 \Leftrightarrow D(x,y) = 0 \Leftrightarrow x = y$.
  Second,
  $\rho(x,y) = D(x,y)^{\beta} = D(y,x)^{\beta} = \rho(y,x)$.
  Finally, we prove the quasitriangle inequality of~$\rho$:
  \begin{eqnarray*}
    \rho(x,y) &=& D(x,y)^{\beta}
     \leq  (D(x,z) + D(z,y))^{\beta} \\
     &\leq&  2^{\beta-1}(D(x,z)^{\beta} + D(z,y)^{\beta})
     = 2^{\beta-1}(\rho(x,z)+ \rho(z,y)).
  \end{eqnarray*}
  So~$\rho$ is a quasimetric, with quasitriangle constant~$A_0 = 2^{\beta-1}$.
\end{proof}

Recall that in Step~3, we fix~$\e = 1/\beta$, and define the metric~$d_{\e}$ from~$\rho$ by the $\e$-chain approach.
Up to now, we have constructed a space of homogeneous type~$(X,\rho,\mu)$ and a metric measure space~$(X,d_{\e},\mu)$.
In Step~4, we claim that the constructed space of homogeneous type~$(X,\rho,\mu)$ and metric measure space~$(X,d_{\e},\mu)$ satisfy the hypotheses of Corollary~\ref{thm.R1.quasimetric}.
To see this, we will show that the metric~$d_{\e}$ coincides with the metric~$D$ in Lemma~\ref{lem:construct2} below. In other words, the metric measure space~$(X,d_{\e},\mu)$ is actually the space~$(X,D,\mu)$ that we started with.

Therefore,  hypotheses~(a) and (e)--(g) of Corollary~\ref{thm.R1.quasimetric} come directly from the assumptions of Example~\ref{thm:construct}.
Also, notice that $(2A_0)^{\e} = (2 \cdot 2^{\beta-1})^{\e} = (2^{\beta})^{1/\beta} = 2$. Thus, by Theorem~\ref{prop:chain_apporoach}, $d_{\e}$~is comparable to~$\rho$.
Then by Lemma~\ref{lem:loccompequi} and~\ref{lem:loccompactequi}, hypotheses~(b) and~(c) hold.

It remains to verify that hypothesis~(d) holds.
For each~$x \in X$, $r>0$ and~$\beta >1$, set
$$\widetilde{B}(x,r) := \{y \in X: D(x,y)^{\beta} <r\} = \{y \in X: D(x,y) <r^{1/\beta}\} = \widehat{B}(x,r^{1/\beta}).$$ Then by assumption~(vi) of Example~\ref{thm:construct}, $\mu(\partial \widetilde{B}(x,r)) = \mu(\widehat{B}(x,r^{1/\beta}))= 0 $, as required. 

To complete the proof of Example~\ref{thm:construct}, our final task is to prove Lemma~\ref{lem:construct2}.
\begin{lem}\label{lem:construct2}
Let~$D$ be a metric on a set~$X$. Suppose~$\beta > 1$. Define~$\rho (x,y) := D(x,y)^{\beta}$ for all~$x, y \in X$. As shown in Lemma~\ref{lem:construct1}, $\rho$ is a quasimetric on~$X$.
Fix~$\e = 1/\beta$.
Let~$d_{\e}$ be the metric defined from~$\rho$ by the $\e$-chain approach.
Then $d_{\e}$ coincides with the original metric~$D$.
\end{lem}
\begin{proof}
We recall the definition of the metric~$d_{\e}$ built via the~$\e$-chain approach:
\[d_{\e}(x,y) = \inf \sum_{i=0}^{n}\rho_{\e}(x_i,x_{i+1}),\]
where the infimum is taken over all finite sequences~$x = x_0,x_1, \ldots, x_n = y$ of points in~$X$.
Using $n =1$, $x_0 = x$ and~$x_1 = y$, this gives us that
\begin{equation}\label{eq1:construct2}
d_{\e}(x,y) \leq D(x_0,x_1) = D(x,y).
\end{equation}
Also, using the triangle inequality for the metric~$D$, for all sequences of points~$x = x_0, x_1, \ldots, x_{k+1} = y$ we get
\begin{eqnarray*}
  D(x,y) = D(x_0, x_{k+1}) &\leq& D(x_0, x_1) + D(x_1, x_{k+1}) \\
   &\leq& D(x_0, x_1) + D(x_1,x_2) + D(x_2, x_{k+1}) \\
   &\vdots& \\
   &\leq& \sum_{i=0}^{k}D(x_i,x_{i+1}).
\end{eqnarray*}
Since~$\e = 1/\beta$, for all~$x, y \in X$ we have
$\rho_{\e}(x,y) = \rho(x,y)^{\e} = \rho(x,y)^{1/\beta} = D(x,y).$
Thus,
\[d_{\e}(x,y) = \inf \sum_{i=0}^{n}\rho_{\e}(x_i,x_{i+1}) = \inf \sum_{i=0}^{n}D(x_i,x_{i+1}).\]
Hence,
\begin{equation}\label{eq2:construct2}
  D(x,y) \leq \inf \sum_{i=0}^{n}\rho_{\e}(x_i,x_{i+1}) = d_{\e}(x,y).
\end{equation}
Combining inequalities~\eqref{eq1:construct2} and~\eqref{eq2:construct2}, Lemma~\ref{lem:construct2} is established.
\end{proof}
This completes the proof of Example~\ref{thm:construct}.
%


\subsection*{Acknowledgements}
We thank Tuomas Hytonen, Ji Li, Jill Pipher, Nageswari Shanmugalingam, and Jeremy Tyson for helpful discussions and comments.

\addcontentsline{toc}{section}{Bibliography}

\end{document}